\documentclass[]{interact}

\usepackage{epstopdf}% To incorporate .eps illustrations using PDFLaTeX, etc.
\usepackage[caption=false]{subfig}% Support for small, `sub' figures and tables
\usepackage{amsmath}
\usepackage{cleveref}

\newtheorem{Definition}{Definition}
\newtheorem{Theorem}{Theorem}
\newtheorem{Remark}{Remark}
\newtheorem{Lemma}{Lemma}

\numberwithin{equation}{section}

\usepackage{siunitx}
\usepackage{amsmath, amssymb}
\usepackage{dsfont}
\usepackage{enumitem}
\usepackage{booktabs}
\usepackage{lmodern}
\usepackage{setspace}
\RequirePackage[round]{natbib}
\DeclareMathOperator{\Var}{Var}
\DeclareMathOperator{\Cov}{Cov}

\newcommand{\mean}[2]{\bar{#1}_{#2}}
\newcommand{\lgamma}{L_{\gamma}}
\newcommand{\lvar}{L_{v}}
\newcommand{\cgamma}{c_{\gamma}}
\newcommand{\cvar}{c_{v}}
\newcommand{\cf}{c_{f}}
\newcommand{\thetaest}{\hat{\theta}_n}
\newcommand{\muhat}{\hat{\mu}_{n, l}}
\newcommand{\xk}{x_{k, n}}
\newcommand{\yk}{y_{k, n}}
\newcommand{\sumBlockSquares}{S_{n, l}}
\newcommand{\sumBlockSquaresMod}{S_{n, l, b}}
\newcommand{\maxzero}[1]{#1_+}

\newcommand{\SDLRD}{$f$-reg-LRD}
\newcommand{\gammaRegLRD}{$\gamma$-reg-LRD}
\newcommand{\gammaIntLRD}{$\gamma$-sum-LRD}
\newcommand{\iesLRD}{IE-LRD}

\usepackage{xcolor}
\definecolor{evgeny}{rgb}{0.25, 0.41, 0.88}
\definecolor{albert}{rgb}{0,0,0}%{0.1215, 0.2275, 0.5765}
\definecolor{marco}{rgb}{1,0.49,0}
\newcommand{\albert}{\textcolor{albert}}

\begin{document}

\articletype{RESEARCH ARTICLE}% Specify the article type or omit as appropriate

\title{Detection of Long Range Dependence in the Time Domain for (In)Finite-Variance Time Series}

\author{
\name{Marco Oesting\textsuperscript{a}, Albert Rapp\textsuperscript{b}\thanks{CONTACT: Albert Rapp, albert.rapp@uni-ulm.de} and Evgeny Spodarev\textsuperscript{b}}
\affil{\textsuperscript{a}Stuttgart Center of Simulation Science, Institute for Stochastics and Applications, University of Stuttgart, Pfaffenwaldring 5a, 70569 Stuttgart, Germany; \\ \textsuperscript{b}Institute of Stochastics, Ulm University, Helmholtzstraße 16, 89069 Ulm, Germany}
}

\maketitle

\begin{abstract}
Empirical detection of long range dependence (LRD) of a time series often consists of deciding whether an estimate of the memory parameter $d$ corresponds to LRD. 
Surprisingly, the literature offers numerous spectral domain estimators for $d$ but there are only a few estimators in the time domain.
Moreover, the latter estimators are criticized for relying on visual inspection to determine an observation window $[n_1, n_2]$ for a linear regression to run on.
Theoretically motivated choices of $n_1$ and $n_2$ are often missing for many time series models. 

In this paper, we take the well-known variance plot estimator and provide rigorous asymptotic conditions on $[n_1, n_2]$ to ensure the estimator's consistency under LRD. 
We establish these conditions for a large class of square-integrable time series models.
This large class enables one to use the variance plot estimator to detect LRD for infinite-variance time series in the sense of indicators of excursion sets.
Thus, detection of LRD for infinite-variance time series is another novelty of our paper.
A simulation study analyzes the LRD detection performance of the variance plot estimator and compares it to a popular spectral domain estimator.
\end{abstract}


\section{Introduction}

\label{sec: introduction}

Long range dependence of a finite-variance stationary time series $X = \{ X(k), k \in \mathbb{Z} \}$ is often characterized by its spectral density $f_X(\lambda)$, $\lambda \in [-\pi, \pi]$.
In particular, a finite-variance stationary process $X$ with spectral density $f_X$ is said to exhibit LRD if
\begin{align*}
f_X(\lambda) = L_f\big(1 / \vert \lambda \vert\big) \vert \lambda \vert^{-2d},
\quad
\albert{\lambda \in [-\pi, \pi],}
\end{align*}
where $L_f(\lambda) \geq 0$ is a slowly varying function and $d \in (0, 0{.}5)$, cf.~Definition 1.2 in \cite{JanBeran.2016}.

Notice that we call a time series stationary if its finite-dimensional distributions are translation-invariant. 
Also, $d \in (-0{.}5,  0{.}5)$ is often referred to as the memory parameter. 
Depending on its value on $(- 0{.}5, 0]$ and on the behavior of $L_f$ near the origin, the standard literature speaks of $X$ as exhibiting short or intermediate range dependence (SRD or IRD) or antipersistence.
In this paper, we treat all cases that correspond to $d \leq 0$ as ``non-LRD''.

In practice, it is of critical importance for the quality of statistical inference whether a time series is LRD.
One effect of LRD is that the sample mean's variance $\Var(\mean{X}{n})$ does not behave as $n^{1/2}$ asymptotically (cf. \Cref{thm: variance-behaviour}) which can lead to e.g. erroneous confidence intervals or hypothesis tests.
For example, Table 1.1 in \cite{JanBeran.2016} depicts simulated rejection probabilities (under the null hypothesis) for the $t$-test at the nominal
5\%-level of significance based on 1000 simulations of a FARIMA$(0, d, 0)$ process with memory parameter $d$.
The findings of that simulation study are that the rejection probabilities rise and the test quality declines as the memory becomes stronger ($d$ larger).

Empirical detection of LRD of a given time series $X$ can be achieved by estimating the memory parameter $d$ and classifying $X$ as LRD when the estimate $\hat{d}$ is larger than zero.
{\color{albert}
In the literature, there are many estimation procedures for the long memory parameters.
Some of them have even been extended to non-stationary processes e.g.~the Whittle estimator \citep{Abadir.2007}. 
For an overview of estimators see \citet[Chapter 5]{JanBeran.2016} or \citet[Chapter 8]{Giraitis.2012}.

In this paper, we focus on two semi-parametric estimation approaches, namely the so-called GPH estimator and the variance plot estimator.
These are based on running an ordinary linear regression on a log-log scale.}
The quantities that are being regressed depend on the perspective one wants to take.
One popular perspective is taken by using a spectral domain approach.
In this case, one uses observations $X(1), \ldots, X(n)$ and regresses the empirical equivalent of $f_X$, i.e. the periodogram 
\begin{align}
\label{eq: peridogram}
I_{n, X}(\lambda) 
:= 
\frac{1}{2\pi n}
\bigg\vert
\sum_{k = 1}^{n} X(k) e^{-ik\lambda} 
\bigg\vert^2, 
\quad
\lambda \in (-\pi, \pi),
\end{align}
against $\lambda$ on a log-log scale.
Probably the most famous form of this procedure uses the first $w \in \mathbb{N}$ smallest Fourier frequencies $\lambda_j = 2 \pi j / n$ and regresses $\log I_{n, X}(\lambda_j)$ against $-2 \log(\lambda_j)$, $j = 1, \dots, w$, in order to estimate $d$ by the regression line slope using the classical least squares method.
This estimator was first proposed by \cite{Geweke.1983} and is thus known as the GPH estimator.
There are various modifications of the GPH estimator like tapered and pooled versions.
These modifications are outside the scope of this paper; we refer to the above-mentioned overview works for more information on them.

Naturally, it is tempting to think that running a linear regression (as in the GPH estimator) poses no problem because the classical least squares estimator is known to be consistent under fairly general conditions.
For example, let $Q_n:= \mathbb{E}{\delta_n\delta_n^T}$ describe the covariance matrix of a linear regression
$
Y_n = Z_n\beta + \delta_n,
$  
where $Z_n$ is a $n\times n$-matrix of deterministic predictor variables and $Y_n$ and $\delta_n$ are $n$-dimensional vectors of response variables and error terms, respectively.
Now, assuming that 
\begin{align}
0 
< 
\inf_{n \in \mathbb{N}} v_{\min}(Q_n)
<
\sup_{n \in \mathbb{N}} v_{\max}(Q_n)
< \infty,
\label{eq: assumption linear regression}
\end{align}
where $v_{\text{min}}(Q_n)$ and $v_{\text{max}}(Q_n)$ are the smallest and largest eigenvalues of $Q_n$, it is known by \citet[Thm.~3.1]{Drygas.1976} that the least squares estimator $\hat{\beta}$ of $\beta$ is consistent if and only if $v_{\text{min}}(Z_n^TZ_n) \rightarrow \infty$ as $n \rightarrow \infty$.
Unfortunately, if the error terms form an LRD random sequence, Theorem 3.3 from \cite{Bottcher.2007} tells us that assumption \eqref{eq: assumption linear regression} is not fulfilled. 
In fact, under LRD conditions, there are many open questions w.r.t. linear regression. To the best of our knowledge, the most comprehensive treatment of linear regression under LRD can be found in \citet[Chapter 11]{Giraitis.2012}. 
There, the authors consider the asymptotic behavior of multiple estimators of $\beta$ when the error terms are given by an LRD Gaussian or linear process. 
In the latter case, the linear process' innovations are usually assumed being independent and having finite second moments.
However, for the estimators that are considered in this paper, it is not clear if any of these conditions on the errors are met.

In the case of the GPH estimator, the regression errors are not asymptotically independent if $X$ is LRD, cf.~\cite[Section 4.6]{JanBeran.2016}.
This makes proving its consistency more complicated in the LRD case.
In fact, the GPH estimator's consistency was established in the original paper \cite{Geweke.1983} only when $X$ is not LRD. 
And it took a couple of years until \cite{Robinson.1995} established asymptotic normality of the GPH estimator under LRD assuming that $X$ is Gaussian.
\cite{Moulines.2003} slightly modified the GPH estimator and were able to lift the latter restriction such that $X$ can be a FARIMA$(p, d, q)$ process whose innovations form an iid sequence with finite fourth moments and a characteristic function that belongs to $\mathcal{L}^r$ for some $r \geq 1$.	
Note that we have used $\mathcal{L}^r$ to denote the space of real-valued functions such that the  $r$-th power of their absolute value is integrable.

Naturally, if there is a semi-parametric estimation procedure in the spectral domain, one expects the existence of a competing estimator in the time domain.
It could use a similar approach as the GPH estimator by replacing the role of the periodogram in a log-log-regression by some counterpart in the time domain.
However, there are only a few treatments of time domain estimators in the literature, e.g.~\cite{TAQQU1995}, \cite{Giraitis1999} or \cite{McElroy2007}.
And even though there are some theoretical results on time domain estimators, many of these estimators are commonly referred to as ``heuristic", cf. p.~416 in \cite{JanBeran.2016}.

In this paper, we consider the so-called variance plot estimator. It is a variance-based estimator in the time domain that estimates the long memory parameter $d$ by regressing $\log \widehat{\Var}(\mean{X}{n})$ against $\log n$, $n = n_1, \dots, n_2$ for suitable cut-off points $n_1, n_2 \in \mathbb{N}$, $n_1 < n_2$.
Here, $\widehat{\Var}(\mean{X}{n})$ is an estimate of $\Var(\mean{X}{n})$, i.e.\ the variance of the sample mean of length $n$.
These cut-off points $n_1$ and $n_2$ fulfill a similar role as the bandwidth parameter $w$ does for the GPH estimator.
In \Cref{THM: CONSISTENCY VAR PLOT}, we establish sufficient conditions on the choice of $n_1$ and $n_2$ such that the slope estimator of the least squares regression line is a consistent estimator of $d$ when the underlying time series is LRD.  
We argue that, this way, it is guaranteed that cut-off points can be chosen such that the variance-based estimator can be used for rigorous conclusions.
Therefore, this easily implemented estimator is more than ``heuristic".
More importantly, our consistency result is true for a large class of models.
This is a crucial advantage compared to other estimators and it is an advantage that we need for the detection of LRD for infinite-variance time series. 

Since we also consider infinite-variance time series, a different definition of LRD will be used in order to talk about the memory of such time series in a meaningful way.
In a simulation study, we compare our variance-based estimator with the modified GPH estimator as stated by \cite{Moulines.2003}.
More specifically, we evaluate how well classifiers based on these estimators are able to correctly classify simulated realizations of fractional Gaussian noise (fGN) and some infinite-variance subordinated fGNs.
Thus, an additional insight of our estimator is demonstrating how already established estimators for finite-variance time series can, in principle, be used to empirically detect LRD of infinite-variance processes.
This relies on transforming the original time series.
However, for most estimators (such as the GPH estimator) it is not clear whether the assumptions of existing limit theorems hold true for the transformed time series.
Due to the large model class for which we proved \Cref{THM: CONSISTENCY VAR PLOT}, our estimator does not have this problem.

The paper is structured as follows.
\Cref{sec: preliminaries} gives an overview of the different definitions of LRD we use throughout this paper.
Further, this section introduces both the GPH estimator and our variance-based estimator and defines (subordinated) fGNs that are used in the simulation study.
In \Cref{sec: consistency}, we prove that our variance-based estimator is consistent given a suitable choice of cut-off points $n_1$ and $n_2$.
In \Cref{sec: simulations}, the findings of the aforementioned simulation study are summarized.  
In \Cref{sec: discussion}, we put our result into the context of the existing literature. 
Also, we discuss open questions. 
In the appendix, Sections \ref{appendix: main result} and \ref{appendix: trafo okay proof} contain the proofs of Theorems \ref{THM: CONSISTENCY VAR PLOT} and \ref{THM: TRAFO IS OKAY}.

\section{Preliminaries}
\label{sec: preliminaries}

In this section, we give preliminary definitions and basic facts serving as the groundwork for the remaining sections.
In \Cref{sec: LRD}, we review multiple non-equivalent definitions of LRD that are used throughout the literature.
In \Cref{sec: fracNoise}, we introduce the (subordinated) fGN which we use in the simulation study in \Cref{sec: simulations}.
Also, ranges of the memory parameter are given for which a (subordinated) fGN is LRD (depending on the notion of LRD).
Finally, \Cref{sec: Estimators} defines the GPH- and the variance-based estimator from the introduction more formally.

\subsection{Long Range Dependence}
\label{sec: LRD}
Before we can define LRD, we will need the well-known notion of regularly varying functions.
\begin{Definition}
	\label{def: slowly varying fct}
	A measurable function $f : [0, \infty) \rightarrow \mathbb{R}$ is called \textit{regularly varying} with index $\alpha \in \mathbb{R}$ if $f$ is either eventually positive or eventually negative and for any $b > 0$ it holds that
	\begin{align*}
	\lim_{x \rightarrow \infty} \frac{f(bx)}{f(x)} = b^\alpha.
	\end{align*} 
	If $\alpha = 0$, $f$ is called \textit{slowly varying}.
\end{Definition}

To differentiate between various definitions of LRD, we will use prefixes in conjunction with the abbreviation LRD.
For instance, LRD as defined by the spectral density's behavior at zero will be referred to as \SDLRD.
Let us begin by recalling Def.~1.2 in \cite{JanBeran.2016}.

\begin{Definition}[\SDLRD{}]
	Let $X = \{ X(k), k \in \mathbb{Z} \}$ be a finite-variance time series with spectral density $f_X(\lambda) = (2\pi)^{-1}\sum_{k \in \mathbb{Z}} \gamma(k) \exp\{ -ik\lambda\}$, where $\gamma$ denotes the auto-covariance function of $X$. 
	Then, $X$ is said to exhibit \textit{\SDLRD{}} if
	\begin{align}
	\label{eq: SD-LRD}
	f_X(\lambda) = L_f\big(1 / \vert \lambda \vert\big) \vert \lambda \vert^{-2d}, \quad \lambda \in [-\pi, \pi],
	\end{align}
	where $d \in (0, 0{.}5)$ and $L_f(\lambda) \geq 0$ is slowly varying.
\end{Definition}

Taking a time-domain (instead of a spectral-domain) perspective leads to definitions of LRD in terms of the autocovariance function $\gamma(k) := \Cov(X(0), X(k))$, $k \in \mathbb{Z}$, of a time series $X$.
However, it is also common to consider not the asymptotic behavior of the autocovariance function $\gamma$ at $\infty$ but its summability.
Thus, two connected but different notions of LRD emerge which can be found in \citet[Eq. (2.1.5) and (2.1.6)]{Pipiras.2017}.

\begin{Definition}[\gammaRegLRD{} and \gammaIntLRD{}]
	Let $X = \{ X(k), k \in \mathbb{Z} \}$ be a finite-variance stationary time series with autocovariance function $\gamma$. 
	Then, $X$ is said to exhibit \textit{\gammaRegLRD{}} if $\gamma$ is regularly varying with exponent $2d -1$ where $d \in (0, 0{.}5)$, i.e.\
	\begin{align}
	\label{eq: gamma-reg-LRD}
	\gamma(k) = L_\gamma(\vert k \vert) \vert k \vert^{2d-1}, \quad k \in \mathbb{Z},
	\end{align}
	where $L_\gamma$ is slowly varying.
	Further, $X$ is said to exhibit \textit{\gammaIntLRD{}} if $\gamma$ is not absolutely summable, i.e.\
	\begin{align*}
	\sum_{k = -\infty}^{\infty} \vert \gamma(k) \vert = \infty.
	\end{align*}
\end{Definition}

All previously mentioned notions of LRD relied on the existence of the process' second moments.
In order to talk about LRD of infinite-variance processes, let us introduce one more notion of LRD.
This notion is based on indicators of excursions and was introduced in \cite{Kulik.2021}.

\begin{Definition}[\iesLRD{}]
	\label{def: iesLRD}
	Let $X = \{ X(k), k \in \mathbb{Z} \}$ be a stationary time series. 
	Then, $X$ is said to exhibit \textit{\iesLRD{}} if there exists a finite measure $\nu$ on $\mathbb{R}$ such that
	\begin{align}
	\label{eq:  iesLRD}
	\sum_{k \in \mathbb{Z}, k \neq 0} \int_{\mathbb{R}} \int_{\mathbb{R}}
	\Big\vert 
	\Cov \Big( \mathds{1}\{ X(0) > u \}, \mathds{1}\{ X(k) > v \} \Big)
	\Big\vert\ \,
	\nu(\mathrm{d}u) \, \nu(\mathrm{d}v) 
	=
	\infty.
	\end{align}
\end{Definition}

\begin{Remark}
	The finiteness of the left-hand side (LHS) of Equation \eqref{eq:  iesLRD} would guarantee the finiteness of \\ $\sum_{k \in \mathbb{Z}, k \neq 0} \int_{\mathbb{R}} \int_{\mathbb{R}}
	\Cov \big( \mathds{1}\{ X(0) > u \}, \mathds{1}\{ X(k) > v \} \big)\ \,
	\nu(\mathrm{d}u) \, \nu(\mathrm{d}v)$ which appears as the limiting variance in a central limit theorem for the $\nu$-averages of the volume of excursion sets.
	This is similar to the motivation of using \gammaIntLRD{} as definition of LRD.
\end{Remark}

Notice that the indicator function is bounded and the measure $\nu$ in \Cref{def: iesLRD} is finite on $\mathbb{R}$.
Therefore, the integral in equality \eqref{eq:  iesLRD} always exists.
Thus, the notion of \iesLRD{} is always defined.

\begin{Remark}
	\label{rem: Independence of marginals}
	\begin{enumerate}[label=(\Roman*)]
		\item Interestingly, \Cref{def: iesLRD} opens up a connection to copula theory which may be used for empirical investigations in the future.
		From Lemma 3.2 in \cite{Kulik.2021} it is known that a stationary time series $X$ whose univariate distribution is absolutely continuous is \iesLRD{} if there exists a finite measure $\nu_0$ on $[0, 1]$ such that
		\begin{align}
		\label{eq: iesLRD copula}
		\sum_{k \in \mathbb{Z}, k \neq 0} \int_{[0, 1]} \int_{[0, 1]}
		\big\vert 
		C_k(u, v) - uv 
		\big\vert \,	
		\nu_0(\mathrm{d}u) \, \nu_0(\mathrm{d}v) 
		=
		\infty,
		\end{align}
		where $C_k$, $k \in \mathbb{Z}$, is the unique copula of the bivariate random vector $(X(0), X(k))$, $k \in \mathbb{Z}$. 
		
		Also, if $\nu_0$ is the Lebesgue measure on $[0, 1]$, then 
		\begin{align*}
		\int_{[0, 1]} \int_{[0, 1]}
		\big\vert 
		C_k(u, v) - uv 
		\big\vert \,
		\mathrm{d}u \, \mathrm{d}v 
		=
		12 \sigma_{C, k},
		\end{align*}
		where $\sigma_{C, k}$ is Schweizer and Wolff's Sigma of $(X(0), X(k))$, cf.~\citet[Eq. 5.3.1]{Nelsen.2010}.
		Additionally, if $X$ is positively associated (cf.~Definition 1.2 in \citealp{Bulinskii.2007}), then $\sigma_{C, k}$  coincides with Spearman's Rho. 
		Possibly, this connection to Spearman's Rho and Schweizer and Wolff's Sigma allows for alternative approaches of detecting \iesLRD{} by empirical estimates of $\sigma_{C, k}$.
		
		\item \label{rem: monotonic trafo} Notice that the definition of \iesLRD{} is invariant under monotonic transformation. This means that given a monotonic transformation $\varphi$ and a time series $X = \{ X(k), k \in \mathbb{Z} \}$, the subordinated time series $Y$ defined by $Y(k) := \varphi(X(k))$, $k \in \mathbb{Z}$, is \iesLRD{} iff $X$ is \iesLRD{}.
		
		\item \albert{
			In practice, complications can arise when one tries to check e.g.~\gammaIntLRD{} using straightforward estimates of the autocovariance functions $\gamma$, cf.~\cite{Hassani.2012}.
			That is why LRD detection often uses intermediate results that arise from the respective LRD conditions.
			In the case of the variance-based estimator, these intermediate results are stated in \Cref{thm: variance-behaviour}.
		}
	\end{enumerate}
\end{Remark}

Further, let us connect the different notions of LRD.
\begin{Remark}
	\label{rem: LRD-Connections}
	\begin{enumerate}[label=(\Roman*)]
		\item Clearly, \gammaRegLRD{} implies \gammaIntLRD{} whereas the opposite is not true in general.
		
		\item If both the spectral density $f$ and the covariance function $\gamma$ exist and are regularly varying, then \SDLRD{} and \gammaRegLRD{} are equivalent. However,
		in general, regular variation of one of the two functions does not guarantee regular variation of the other one. Therefore, additional assumptions are often imposed.

		\item Theorem 6.2.11 and Remark 6.2.12 in \cite{Samorodnitsky.2016} give sufficient conditions under which \gammaRegLRD{} and \SDLRD{} are equivalent. If one assumes that $\gamma(n)$, $n > 0 $, is eventually non-increasing, \gammaRegLRD{} implies \SDLRD{}.
		The converse is true if the spectral density $f$ is of bounded variation on an interval $(a, \pi)$, where $a \in (0, \pi)$ and $f(\lambda)$, $\lambda > 0$, is non-increasing in a neighborhood of the origin.
		
		\item \albert{Since the autocovariance function $\gamma_X$ and the spectral density $f_X$ are Fourier transforms of each other, the relationship between \eqref{eq: SD-LRD} and \eqref{eq: gamma-reg-LRD} can be viewed in the context of Abelian and Tauberian theorems.
		These theorems are discussed e.g.~in Sections 4.3 and 4.10 of \cite{Bingham.1987}.}
		
		\item \label{rem: LRD-Connections pos. associated} If $\Cov \big( \mathds{1}\{ X(0) > u \}, \mathds{1}\{ X(k) > v \} \big)$ is either non-negative or non-positive for all $u, v \in \mathbb{R}$ and $k \in \mathbb{Z}$, the absolute value on the LHS of \eqref{eq:  iesLRD} can be omitted. This holds e.g. if $X$ is positively or negatively associated, cf.~Definition 1.2 in \cite{Bulinskii.2007}.
		In this case, due to Fubini's theorem, for any finite measure $\nu$, the integrals on the LHS of \eqref{eq:  iesLRD} can be rewritten as
		\begin{align*}
		&\int_{\mathbb{R}} \int_{\mathbb{R}}
		\Cov \Big( \mathds{1}\{ X(0) > u \}, \mathds{1}\{ X(k) > v \} \Big)\, \nu(\mathrm{d}u) \, \nu(\mathrm{d}v) \\
		&=
		\Cov
		\bigg( 
		\int_{\mathbb{R}} \mathds{1}\{ X(0) > u \}\, \nu(\mathrm{d}u),
		\int_{\mathbb{R}} \mathds{1}\{ X(k) > v \}\, \nu(\mathrm{d}v)
		\bigg) 
		=:
		\Cov(Y_\nu(0), Y_\nu(k)).
		\end{align*} 
		Therefore, such a time series $X$ is \iesLRD{} iff there exists a finite measure $\nu$ on $\mathbb{R}$ such that the transformed time series $Y_\nu(k) := \int_{\mathbb{R}} \mathds{1}\{ X(k) > u \}\, \nu(\mathrm{d}u)$ is \gammaIntLRD.		
	\end{enumerate}
\end{Remark}

Next, let us collect a few well-known results that will motivate our variance-based estimator.
Recall from \Cref{def: slowly varying fct} that a slowly varying function can be eventually negative. 
This ensures that, in the next theorem, the sign of $v(d)$ matches that of the slowly varying function $\lgamma$ such that the asymptotic behavior of the sample mean's variance is well-defined.

\begin{Theorem}
	\label{thm: variance-behaviour}
	Let $X = \{ X(k), k \in \mathbb{Z} \}$ be a finite-variance time series with autocovariance function $\gamma(k) = \lgamma(k) \vert k \vert^{2d - 1}$, $k \in \mathbb{Z}$, where $\lgamma$ is a slowly varying function at infinity and $d \in (-0{.}5, 0{.}5)$.
	Also, let $\mean{X}{n}$ denote the sample mean of $X(1), \ldots, X(n)$, $n \in \mathbb{N}$. 
	\begin{enumerate}[label=(\Roman*)]
		\item If either $d < 0$ and $\sum_{k \in \mathbb{Z}} \gamma(k) = 0$, or $d > 0$, there exists a slowly varying function $L_v$ such that $\Var(\mean{X}{n}) = L_v(n) n^{2d - 1}$ as $n \rightarrow \infty$.
		More precisely, $L_v(n) \sim v(d) L_\gamma(n)$ as $n \rightarrow \infty$ where $v(d) := 1 / (d (2d + 1))$.	
		\item If $d < 0$ and $\sum_{k \in \mathbb{Z}} \gamma(k) \in (0, \infty)$, then  $\Var(\mean{X}{n}) \sim  \Big( \sum_{k \in \mathbb{Z}} \gamma(k) \Big) n^{-1}$.
	\end{enumerate}
	
\end{Theorem}
\begin{proof}
	\begin{enumerate}[label=(\Roman*)]
		\item This follows from the first part of the proof of Corollary 1.2. in \cite{JanBeran.2016}. 
		\item Straightforward calculations and the use of dominated convergence yield
		\begin{align*}
		n \Var(\mean{X}{n}) 
		= 
		\gamma(0) + 2 \sum_{k = 1}^{n - 1} \bigg( 1 - \frac{k}{n} \bigg) \gamma(k)
		\rightarrow
		\gamma(0) + 2 \sum_{k = 1}^{\infty} \gamma(k) 
		= \sum_{k \in \mathbb{Z}} \gamma(k)
		\end{align*}
		as $n \rightarrow \infty$.
	\end{enumerate}
\end{proof}

\begin{Remark}
	\label{rem: behaviour of variance}
	\begin{enumerate}[label=(\Roman*)]
		\item The case $d = 0$ is a rather special case as it opens up a discussion on whether this value should correspond to ``LRD" or ``SRD" or even intermediate dependence. 
		Depending on the definition of LRD, this may yield different results.
		For example, even though the case $d = 0$ is not covered in \Cref{thm: variance-behaviour}, by monotonicity arguments the theorem suggests that the asymptotic behavior of $\Var(\mean{X}{n})$ should get arbitrarily close to $n^{-1}$ which may imply ``SRD" as $n^{-1}$ is the same asymptotic behavior as for independent random variables. Yet, $d = 0$ also implies that the covariances are not absolutely summable, i.e. \gammaIntLRD.
		Thus, for simplicity, \Cref{thm: variance-behaviour} ignores this case.
		
		\item The variance plot tries to estimate $2d - 1$ via the regression line slope of $\log (\Var(\mean{X}{n}))$ on $\log n$ for which \Cref{thm: variance-behaviour} delivers the theoretical foundation.
		Clearly, \Cref{thm: variance-behaviour} establishes that a slope of $-1$ is a threshold value to differentiate between \gammaRegLRD{} and not \gammaRegLRD. 
		
		\item Notice that if $d < 0$ and $\sum_{k \in \mathbb{Z}} \gamma(k) > 0$, then the variance of $\mean{X}{n}$ cannot decrease faster than $n^{-1}$. Consequently, a reliable estimation of $d$ from that variance is not possible in this case anymore.
		For the purpose of classification, however, the threshold value $-1$ remains a useful guide to distinguish between LRD and SRD.
		
		\item \label{item: cvar and cgamma} If $\lgamma(n) \rightarrow \cgamma$ as $n \rightarrow \infty$ where $\cgamma \neq 0$ is a constant that depends on $\gamma$ (and consequently $d$ as well), \Cref{thm: variance-behaviour} establishes that the slowly varying function $\lvar$ converges to a positive constant $\cvar$ as $n \rightarrow \infty$.
		
		\item \albert{
			Just like in \Cref{thm: variance-behaviour}, the case $d \in (-0{.}5,  0{.}5)$ is commonly investigated in the LRD literature.
			It is worth pointing out that there are time series with autocovariance function $\gamma(k) = \lgamma(k) \vert k \vert^{2d - 1}$, $k \in \mathbb{Z}$, where $d \in (-1, -0{.}5)$, cf.~\cite{Bondon.2007}.
			Hence, the choice $d \in (-0{.}5,  0{.}5)$ may be too restrictive in general.
			However, since we are mainly interested in $d > 0$, we omit a more detailed discussion and stick to the commonly used case $d \in (-0{.}5,  0{.}5)$.
		}
	\end{enumerate}
\end{Remark}
\subsection{Fractional Gaussian Noise and Subordinated Gaussian Processes}
\label{sec: fracNoise}

{\color{albert}
In the LRD literature, it is common to investigate time series of the form $X_t = \sum_{j = 0}^\infty a_j \varepsilon_{t - j}$ where $(\varepsilon_j)_{j \in \mathbb{Z}}$ is a sequence of innovations.
Specificically, fractionally integrated processes belong to this class of time series, cf.~Chapters 5 \& 6 in \cite{Hassler.2018} or Chapter 7 in \cite{Samorodnitsky.2016}.
However, as \cite{Hassler.Hosseinkouchack.2020b} and \cite{Hassler.Hosseinkouchack.2020a} pointed out, the harmonically weighted times series that uses $a_j = 1 / (1 + j)$ fulfills that $\Var(\mean{X}{n}) \sim C \frac{\log^2 n}{n}$  where $C > 0$ is a constant, but it is hard to empirically detect this time series as SRD.

To avoid such notoriously tricky time series for the comparison of our estimators in \Cref{sec: simulations}, we consider another well-understood class of time series from the LRD literature.
Namely, we consider the fractional Brownian motion and the fractional Gaussian noise.
More information on these processes beyond what we are going to introduce here can be found e.g.\ in Section 2.6 in \cite{Pipiras.2017}.
}
\begin{Definition}
	\label{def: lfsm}
	Suppose $H \in (0, 1)$.
	Further, define a kernel function $g_t$ by
	\begin{align*}
	g_t(H, x) =
	\maxzero{(t-x)}^{H - 1/2} 
	- 
	\maxzero{(-x)}^{H - 1/2},
	\quad
	x, t \in \mathbb{R},
	\end{align*}
	for $H \neq 1 /2$ and
	\begin{align*}
	g_t(1 / 2, x) = \begin{cases}
	\mathds{1}\big\{ x \in [0, t] \big\}, &t \geq 0, x \in \mathbb{R},\\
	\mathds{1}\big\{ x \in [t, 0] \big\}, &t < 0, x \in \mathbb{R},
	\end{cases} 
	\end{align*}
	where $\maxzero{x} = \max\{ x, 0 \}$.
	Then, the so-called \textit{fractional Brownian motion (fBM) with Hurst index $H$} is a stochastic process $X = \{ X(t), t \in \mathbb{R} \}$ such that
	\begin{align*}
	X(t) = \frac{\sigma^2}{C(H)}\int_{\mathbb{R}} g_t(H, x) \, \Lambda(\mathrm{d}x),
	\end{align*}
	where $\Lambda$ is a standard Gaussian random measure with Lebesgue control measure and $C^2(H) = \int_\mathbb{R} g_1^2(H, x)\, \mathrm{d}x$ is a normalizing constant such that $\Var(X(1)) = \sigma^2$. 
	Furthermore, the increment process $Y(k) = X(k) - X(k - 1)$, $k \in \mathbb{Z}$, is stationary and known as \textit{fractional Gaussian noise (fGN) with Hurst index $H$}.
\end{Definition}

\begin{Remark}
	\label{rem: lfsm}
	\begin{enumerate}[label=(\Roman*)]
		\item Note that it is common in the literature to refer to a fBM's memory despite its non-stationarity.
		What is usually meant is the memory of the corresponding stationary increment process, i.e.\ the corresponding fGN.
		
		\item The covariance function $\gamma_Y$ and spectral density $f_Y$ of a fGN $Y = \{ Y(k), k \in \mathbb{Z} \}$ are well-known. For instance, Proposition 2.8.1\ in \cite{Pipiras.2017} yields
		\begin{align}
		\gamma_Y(k) &= \frac{\sigma^2}{2}\Big( 
		\vert k + 1 \vert^{2H} +
		\vert k - 1 \vert^{2H} -
		2 \vert k \vert^{2H}
		\Big) , \quad k \in \mathbb{Z},
		\label{eq: covariance function fGN} \\
		f_Y(\lambda) &=  \frac{\sigma^2 \Gamma(2H + 1) \sin(H\pi)}{2\pi} 
		\vert 1 - e^{-i\lambda} \vert^2 
		\sum_{n = -\infty}^\infty 
		\vert \lambda + 2 \pi n \vert^{-1 -2H}, \quad \lambda \in (-\pi, \pi). 
		\label{eq: spectral density fGN} 
		\end{align}
		Additionally, for $H \neq \frac{1}{2}$, it holds that 
		\begin{align}
		\gamma_Y(k) &\sim \sigma^2 H (2H - 1) k^{2H - 2}
		\label{eq: asymp covariance}, \quad k \rightarrow \infty, \\
		f_Y(\lambda) &\sim \frac{\sigma^2 \Gamma(2H + 1) \sin(H\pi)}{2\pi} \vert \lambda \vert^{1 - 2H},
		\quad \lambda \rightarrow 0
		\label{eq: asymp spectral density}
		\end{align}
		Consequently, the fGN $Y$ is \gammaRegLRD{} and \SDLRD{} iff $H > 1/2$.
		
	\end{enumerate}
\end{Remark}

As the variance of a fGN $Y$ is finite by definition, the notions of \gammaRegLRD{} and \SDLRD{} apply to it naturally. 
For the sake of investigating LRD of infinite-variance processes, let us introduce \textit{subordinated Gaussian processes}.
These are defined as processes $Z$ that fulfill $Z(t) = G(Y(t))$ for all $t \in \mathbb{R}$ where $G$ is a measurable function and $Y$ is a Gaussian process. 

In our simulation study, we will consider
\begin{align}
Z(k) = e^{Y^2(k) /  (2 \alpha)}, \quad k \in \mathbb{Z}, \label{eq: subordinated Gaussian}
\end{align}
where $Y$ is a fGN with Hurst index $H$, variance $\sigma^2>0$ and $\alpha > 0$.
Clearly, this subordinated Gaussian process has an infinite variance iff $\alpha \leq 2 \sigma^2$.
Therefore, the notions of \gammaRegLRD{} and \SDLRD{} do not apply anymore. 
Thus, we will investigate this process' memory in terms of \iesLRD{}.

\begin{Theorem}
	Let $Y$ be a fGN with Hurst index $H$ and $Z(k) = e^{Y^2(k) /  (2 \alpha)}$, $\alpha > 0$, $k \in \mathbb{Z}$. Then, $Z = \{Z(k), k \in \mathbb{Z} \}$ is \iesLRD{} iff $H \geq \frac{3}{4}$.
\end{Theorem}

\begin{proof}
	By \Cref{rem: Independence of marginals}\ref{rem: monotonic trafo}, we can assume  w.l.o.g.\ that $Y$ is a unit-variance fGN.
	Furthermore, Example 3.9 in \cite{Kulik.2021} states that for every stationary Gaussian process $Y$ with covariance function $\gamma_{Y}$ such that $\vert \gamma_{Y}(k) \vert \leq 1$ for all $k \in \mathbb{Z}$ and $\gamma_{Y}(k) \sim \vert k \vert^{-\eta}$, $\eta > 0$, as $k \rightarrow \infty$, it holds that $Z$ is \iesLRD{} iff $\eta \leq \frac{1}{2}$.
	Finally, asymptotic relation \eqref{eq: asymp covariance} gives us $\eta = 2 - 2H \leq 1 / 2$ iff $H \geq 3/4$. 
\end{proof}

\begin{Remark}
	Recall from  \Cref{rem: Independence of marginals}\ref{rem: monotonic trafo} that the notion of \iesLRD{} is invariant under monotonic transformations.
	Interestingly, the subordinated Gaussian process \eqref{eq: subordinated Gaussian} is a monotonic transformation of the second Hermite polynomial $H_2(X(k)) = X(k)^2 - 1$, $k \in \mathbb{Z}$. 
	But for a subordinated fGN $H_q(X)$ where $H_q$, $q \geq 2$, is the $q$-th Hermite polynomial, it is well-known that its normalized partial sums converge to a Gaussian random variable iff $H < \frac{2q - 1}{2q}$, cf.\ \cite{Taqqu.1975} and \cite{Breuer.1983} for the original convergence result and Theorem 4.1 in \cite{Nourdin.2009} for Berry-Esse\'en bounds. 
	
	Consequently, the memory threshold $H = \frac{3}{4}$ fits well into the existing theory and even has a connection to limit theorems with non-Gaussian limits dealing with another definition of LRD in the classical finite-variance literature. 
	Namely, Chapter 9 in \cite{Samorodnitsky.2016} reasons that LRD occurs after a phase transition in parameters of the model under which the limit of a statistic of interest significantly changes. 
	Of course, this notion is also applicable to infinite-variance processes.
	From this perspective, $H = \frac{3}{4}$ marks the point of phase transition from Gaussian to non-Gaussian limits in our case.
	
	{ \color{albert}
	Of course, this notion of phase transition depends on the statistic that is used to investigate long memory.
	For example, we know that the notion of \iesLRD{} considers long memory through excursion sets and is invariant under strictly monotonic transformations.
	Due to this invariance, we were able to consider the second Hermite polynomial $H_2(X)$ of our fGN $X$ in the above discussion.
	But other definitions of LRD may not be invariant under strictly monotonic transformation and one has to consider the subordinated Gaussian process as defined in \eqref{eq: subordinated Gaussian}.
	Using the findings from \cite{Sly.2008}, one can see that the properly normalized partial sum process of \eqref{eq: subordinated Gaussian} with $\alpha < 2$ converges to an $\alpha$-stable Lévy motion for $H < \frac{1}{2} + \frac{1}{2\alpha}$ or to the second Hermite process for $H > \frac{1}{2} + \frac{1}{2\alpha}$ .
	Clearly, this marks a different point of phase transistion and would correspond to long memory in some other sense.
	In the end, this different notion of long memory relates to a different statistic, namely the sample mean, that is used to investigate the memory. 
	}
\end{Remark}

Finally, let us show that the subordinated Gaussian process \eqref{eq: subordinated Gaussian} fulfills the conditions of \Cref{rem: LRD-Connections}\ref{rem: LRD-Connections pos. associated}.
\begin{Lemma}
	\label{thm: cov subordinated fGN}
	Let $Y$ be a stationary, zero mean Gaussian process.
	Define $Z$ via $Z(k) = e^{Y^2(k) /  (2 \alpha)}, k \in \mathbb{Z}$.
	Then it holds that $\Cov \big( \mathds{1}\{ Z(0) > u \}, \mathds{1}\{ Z(k) > v \} \big) \geq 0$ for all $k \in \mathbb{Z}$ and $u, v \in \mathbb{R}$.
\end{Lemma}

\begin{proof}
	First, notice that for $u < 1$ or $v < 1$ it holds that $\Cov \big( \mathds{1}\{ Z(0) > u \}, \mathds{1}\{ Z(k) > v \} \big) = 0$.
	For $u, v \geq 1$, we use symmetry arguments to compute
	\begin{align}
	&\Cov \big( \mathds{1}\{ Z(0) > u \}, \mathds{1}\{ Z(k) > v \} \big) \notag \\
	&= 
	\Cov \big( \mathds{1}\{ \vert Y(0) \vert > \sqrt{2\alpha \log u} \}, \mathds{1}\{ \vert Y(k) \vert > \sqrt{2\alpha \log v} \} \big) \notag \\
	&=
	2 \Big(
	\Cov \big( \mathds{1}\{ Y(0) > \tilde{u} \}, \mathds{1}\{ Y(k) > \tilde{v} \} \big) 
	+ \Cov \big( \mathds{1}\{ Y(0) > \tilde{u} \}, \mathds{1}\{ -Y(k) > \tilde{v} \} \big)
	\Big),
	\label{eq: sum of covariance of inds}
	\end{align}
	where $\tilde{u} := \sqrt{2\alpha \log u}$, $\tilde{v} := \sqrt{2\alpha \log v} \geq 0$.
	A Gaussian random vector $(U, V)$ with zero means, variances $\sigma^2$ and correlation coefficient $\rho$ satisfies
	\begin{align*}
	\Cov \big( &\mathds{1}\{ U > u \}, \mathds{1}\{ V > v \} \big) 
	=
	\frac{1}{2\pi} \int_{0}^{\rho} 
	\frac{1}{\sqrt{1 - r^2}} \exp \bigg\{ -
	\frac{u^2 - 2ruv + v^2}{2\sigma^2(1 - r^2)}
	\bigg\}
	\, 
	\mathrm{d}r,
	\end{align*}
	cf. e.g.\ Lemma 2 in \cite{Bulinski.2012}.
	Consequently, the sum in \eqref{eq: sum of covariance of inds} equals
	\begin{align*}
	\frac{1}{2\pi} \int_{0}^{\vert \gamma_{Y}(k) / \gamma_{Y}(0) \vert}
	&\frac{1}{\sqrt{1 - r^2}} 
	\exp \bigg\{ -
	\frac{\tilde{u}^2 + \tilde{v}^2}{2\sigma^2(1 - r^2)}
	\bigg\} \\
	&\cdot
	\bigg(
	\exp \bigg\{
	\frac{r\tilde{u}\tilde{v}}{\sigma^2(1 - r^2)}
	\bigg\}
	-
	\exp \bigg\{ -
	\frac{r\tilde{u}\tilde{v}}{\sigma^2(1 - r^2)}
	\bigg\}
	\bigg)
	\mathrm{d}r.
	\end{align*}
	Finally, the claim follows from $e^x - e^{-x} = 2 \sinh(x) \geq 0$ for $x \geq 0$.
\end{proof}

\subsection{Memory Parameter Estimators}
\label{sec: Estimators}

Let us introduce the competing estimators.
We start with the so-called GPH estimator in \Cref{sec: GPH Estimator} as a popular estimation procedure.
Then, in \Cref{sec: variance plot} we introduce the variance plot estimator from the literature, its critique and offer improvements which lead to the consistency of the estimator.

\subsubsection{GPH Estimator}
\label{sec: GPH Estimator}

As was already mentioned, the GPH estimator can be thought of as a similar approach to the variance plot in the spectral domain (see \citealp[Chapter 5.6.2]{JanBeran.2016} for more information).
Assume that $X$ is a stationary, \SDLRD{} process with spectral density
\begin{align}
f_X(\lambda) \sim \cf \vert \lambda \vert ^{-2d},
\quad \lambda \rightarrow 0,
\label{eq: spectral density assumption}
\end{align}
where $d \in (0, 1/2)$ and $\cf \neq 0$.
Consequently, for $b(\lambda) := -2 \log \vert \lambda \vert$, it holds that
\begin{align}
\big\vert 
\log f_X(\lambda) 
- \big(\log c_f + d \cdot b(\lambda)\big) \big\vert
\rightarrow 0, \quad \lambda \rightarrow 0.
\label{eq: spectral density log asymptotic}
\end{align}

Again, the empirical counterpart of the spectral density is given by the periodogram which we defined in equation \eqref{eq: peridogram}.
Now, the asymptotic relation \eqref{eq: spectral density log asymptotic} motivates approximating $\log I_{n, X}(\lambda)$ through a linear regression of the form
$
\log I_{n, X}(\lambda) 
= 
\beta_0 + \beta_1 b(\lambda) + \delta,
$
where $\delta$ describes the approximation error.
Subsequently, using the $w$ smallest Fourier frequencies $\lambda_k = \frac{2 \pi k}{n}$, $k = 1, \dots, w$, the memory parameter $d$ can be estimated via the standard least squares slope estimator
\begin{align*}
\hat{d}_{\text{GPH}} = \frac{\sum_{k = 1}^{w} (b_k - \mean{b}{w}) \log I_{n, X}(\lambda_k)}{\sum_{k = 1}^{w} (b_k - \mean{b}{w})^2},
\end{align*}
where $b_k = -2 \log(\lambda_k)$ and $\mean{b}{w}$ describes the mean of $b_1, \dots, b_w$.
In this case, $w$ is called the bandwidth parameter and needs to be chosen such that $w \rightarrow \infty$  and $w / n \rightarrow 0$ as $n \rightarrow \infty$.
\albert{
As mentioned in the introduction, the original publication by \cite{Geweke.1983} did not show consistency of this estimator when the underlying process $X$ was in fact \SDLRD.
But under the additional assumption that the time series is Gaussian, \cite{Hurvich.1998} proved that the estimator is asymptotically normal under \SDLRD{}.
Furthermore, \cite{Robinson.1995} and \cite{Moulines.2003} proved  that for a refined version of the GPH estimator which also trims some low Fourier frequencies, asymptotic normality holds under \SDLRD{} even without the Gaussianity assumption.
}
More specifically, given an integer $0 < l < w$ one can  rewrite the least squares estimators of $d$ as
\begin{align*}
\hat{d}_{\text{GPH}}(l) := \frac{\sum_{k = 1}^{N} (b_{k, l} - \mean{b}{N, l}) \log I_{n, X}(\lambda_{k, l})}{\sum_{k = 1}^{N} (b_{k, l} - \mean{b}{N, l})^2},
\end{align*}
where $N = w - l + 1$, $\lambda_{k, l} = \frac{2\pi(l + k - 1)}{n}$ and $\mean{b}{N, l}$ is the mean of $b_{k, l} = -2\log(\lambda_{k, l})$, $k = 1,\ldots, N$. 

Moreover, it is known that under somewhat complex conditions on the behavior of the spectral density at zero it holds that 
$
\sqrt{w}
\big(
\hat{d}_{\text{GPH}}(l) - d
\big)
\stackrel{\text{d}}{\rightarrow}
\mathcal{N}(0, \pi^2 / 24).
$
Note that \cite{Robinson.1995} and \cite{Moulines.2003} also give asymptotic conditions on how the parameters $l$ and $w$ need to be chosen w.r.t.\ to the sample length. 
However, practical guidance for their choices is missing for finite sample size.

\subsubsection{Variance Plot}
\label{sec: variance plot}

\albert{In this section, let us introduce the variance-type estimator. 
It was first considered e.g.~in \cite{Teverovsky.1997} or \cite{Giraitis1999}.}
Let $X = \{ X(k), k = 1, \dots, n \}$ be a sample of a time series with covariance function $\gamma(n) = \lgamma(n) n^{2d-1}$, $d \in (-0{.}5, 0{.}5) \setminus \{0\}$, and $\lgamma$ being a slowly varying function which converges to a constant $\cgamma \neq 0$ as $n \rightarrow \infty$, i.e.
\begin{align}
\lgamma(n) \rightarrow \cgamma \neq 0
\label{eq: slowly varying convergence assumption}
\end{align}
as $n \rightarrow \infty$. Now, from \Cref{thm: variance-behaviour} it follows that
\begin{align}
\Var(\mean{X}{n}) = \lvar(n) n^{2D-1},
\label{eq: var consequence}
\end{align} 
where $\lvar$ is a slowly varying function that converges to a positive constant $\cvar$ as $n \rightarrow \infty$, $D = 0$ if $d <0$ and $\sum_{n \in \mathbb{Z}} \gamma(n) > 0$, and $D = d$ otherwise.

Similar to what was done for the GPH estimator, Equation \eqref{eq: var consequence} motivates
\begin{align}
\big\vert 
\log \Var(\mean{X}{n}) - \big(\log c_v + (2D - 1) \log n\big) 
\big\vert 
\rightarrow 0,
\quad
n \rightarrow \infty.
\end{align}
Thus, one can estimate $2D - 1$ by estimating the slope of a linear regression of $\log \Var(\mean{X}{n})$ on $\log n$.
To do this empirically,  one first estimates $\Var(\mean{X}{l})$ for $l = 1, \dots, n,$ by
\begin{align*}
S_l^2 := \frac{1}{n - l + 1} \sum_{k = 1}^{n - l + 1} (\mean{B}{k,l} - \muhat)^2,
\end{align*}
where $\mean{B}{k,l}$, $k = 1, \dots, n - l + 1$, denotes the mean of the block $(X(k), \dots, X(k + l - 1))$ and $\muhat$ is the sample mean of all block means $\mean{B}{k,l}$.
Then, $\theta = 2D - 1 \in (-2,0)$ is estimated as the regression slope  $\thetaest$ of $\log S_l^2$ on $\log l$ for $l = 1, \dots, n$ by least squares.

Further, the regression usually needs to be computed based on block lengths $l = n_1, \dots, n_2,$ where $n_1, n_2 \in \mathbb{N}$ with $n_1 < n_2$.
In summary, the slope estimator $\thetaest$ is written 
\begin{align}
\thetaest = \frac{\sum_{k = 1}^{N}(\xk - \mean{x}{N})(\yk - \mean{y}{N})}{\sum_{k = 1}^{N}(\xk - \mean{x}{N})^2}
\label{eq: slope estimator variance}
\end{align}
with $N = n_2 - n_1 + 1$, $\xk = \log(n_1 + k - 1)$, $\yk = \log S^2_{n_1 + k - 1}$, $k = 1, \dots, N$.
Also, $\mean{x}{N}$ and $\mean{y}{N}$ represent the means of $\xk$ and $\yk$, $k = 1, \dots, N$, respectively.
Then, the time series can be classified as \gammaRegLRD{} if $\hat{\theta}_n > -1$.

\begin{Remark}
	\begin{enumerate}[label=(\Roman*)]
		\item \albert{
		This variance-type estimator makes use of popular block resampling techniques, see e.g.~\cite{Kim.2011} or \cite{Zhang.2022}.
		This helps us to establish our main result, \Cref{THM: CONSISTENCY VAR PLOT}, by making use of proof techniques from \cite{Kim.2011}.
		}
				
		\item As mentioned before, the theoretically motivated choice of the observation window $[n_1, n_2]$ for the variance plot is still an open problem for many models. 
		In \Cref{THM: CONSISTENCY VAR PLOT}, we establish a valid asymptotic range of $n_1$ and $n_2$ for a large class of models.  
		As we will see, our choice leads to consistency of the slope estimator $\thetaest$.
		
		\item Notice that we do not use disjoint blocks $B_{k,l}$, $k = 1, \dots, n - l + 1$, to estimate the sample mean's variance.
		This is in contrast with the common procedure for the variance plot as, for instance, in \cite[Chapter 5.4.1]{JanBeran.2016}.
		Naturally, this deviation may feel unintuitive because one would expect disjoint blocks to have nicer properties since they may be thought of as ``closer to independence".
		However, overlapping blocks are more efficient in the sense of Eq.~(3.46) in \cite{Politis.1999}.
	\end{enumerate}
\end{Remark}

\section{Consistency under LRD}
\label{sec: consistency}

Let us state the main result of this paper.
It is valid for finite-variance time series and its proof can be found \Cref{appendix: main result}.
For the infinite-variance case we will perform a suitable transformation such that we can investigate LRD in terms of \Cref{def: iesLRD}.

Also, let us point out that our result will only consider non-deterministic time series. 
Heuristically speaking, a time series $X$ is non-deterministic if its values $X(t)$ at time points $t \in \mathbb{Z}$ are \textit{not} perfectly linear predictable by the observed past $X(s)$, $s \leq t$.
For a formal definition of non-deterministic time series let us refer to e.g. Chapter 5.7 in \cite{Brockwell.1991} and \Cref{def: determistic} below.
Notice that the notion of non-deterministic time series is often used in the context of weakly stationary time series. 
These are time series whose mean function is constant and whose autocovariance function depends only on the temporal lag.

\begin{Definition} \label{def: determistic}
	Let $X = \{ X(k), k \in \mathbb{Z} \}$ be a weakly stationary time series. Further, define $\mathcal{M}_n(X) = \overline{\text{span}} \{X(k),\ k \leq n \}$, the closure of the linear subspace spanned by $X$ up to time point $n$. Then, $X$ is said to be \textit{deterministic} if $X(n + j)$, $j \in \mathbb{N}$, is perfectly predictable in terms of elements of $\mathcal{M}_n$. This is equivalent to $\mathbb{E}\big[ \vert X(n+1) - P_{\mathcal{M}_n} X(n + 1) \vert^2\big] = 0$ where $P_{\mathcal{M}_n} X(n + 1)$ denotes the projection of $X(n + 1)$ onto $\mathcal{M}_n=\mathcal{M}_n(X)$.
	We call $X$ \textit{non-deterministic} if $X$ is not deterministic.
\end{Definition}

\begin{Theorem}
	\label{THM: CONSISTENCY VAR PLOT}
	Let $X = \{ X(k), k \in \mathbb{Z} \}$ be a stationary, non-deterministic time series whose spectral density exists.
	Further, assume that the autocovariance function $\gamma(k) = \lgamma(k) \vert k \vert^{2d-1}$, $d \neq 0$, of $X$ fulfils condition \eqref{eq: slowly varying convergence assumption}.
	Also, let $\theta \in (-2, 0)$ be the index of regular variation of $\Var(\mean{X}{n}) =  \lvar(n) n^{\theta}$ that arises in \Cref{thm: variance-behaviour}.	
	Then, for $n_1 = n^{\delta}$ and $n_2 = m n_1$ where  $0 < \delta < \min \Big\{ \frac{2 \vert \theta \vert}{4 \vert \theta \vert + 1} ,  \frac{\vert \theta \vert}{\vert \theta \vert + (|\theta|-1)_+ + 1}\Big\}$ and $m > 1$ it holds that $\vert \thetaest - \theta \vert \stackrel{\text{P}}{\rightarrow} 0$ as $n \rightarrow \infty$ with $\thetaest$ as given in Equation \eqref{eq: slope estimator variance}.
\end{Theorem}

\begin{Remark}
	\label{rem: after main theorem}
	\begin{enumerate}[label=(\Roman*)]
		\item Notice that \Cref{THM: CONSISTENCY VAR PLOT} assumes that the cutoff parameter $\delta$ needs to be chosen dependent on $\theta$. 
		Consequently, it is still an open question as to how $m$ and $\delta$ can be "optimally" chosen in some sense.
		Nevertheless, \Cref{THM: CONSISTENCY VAR PLOT} proves the existence of lower and upper cutoff bounds such that the estimator is consistent.
		Similar statements are true for the GPH estimator, cf. Remark \ref{rem: after main theorem}\ref{rem: parameter choice for GPH}.
		In \Cref{sec: simulations}, we will see that, in practice, multiple choices of cutoff bounds can lead to classification results that are comparable to the GPH estimator.
		
		\item \label{rem: parameter choice for GPH}In the simulation study in \Cref{sec: simulations}, we will compare classification results of our estimator and the GPH estimator. 
		Thus, let us mention that an optimal choice of the GPH estimator's cutoff parameter $l$ and $w$ is not obvious either.
		There is a result that gives an optimal choice for $w$ when $l = 0$ is assumed, cf.~\citet[Section 5.6.4.5]{JanBeran.2016}. However, this optimal choice depends on unknown parameters as well. Some suggested alternatives exist in the literature, e.g. \cite{Hurvich1994} and \cite{Hurvich1999}. Since we are only interested in comparability of our estimator, we will simply run a grid search and test all cutoffs. 
		
		\item A sufficient condition for the existence of the spectral density of $X$ in \Cref{THM: CONSISTENCY VAR PLOT} is given by Theorem 6.2.11.~and Remark 6.2.12.~in \cite{Samorodnitsky.2016}. These results state that the spectral density exists if the autocovariance function is eventually non-increasing.
		
		\item A stationary time series with spectral density $f$ is non-deterministic iff $\int_{-\pi}^{\pi} \log f(x)\, \mathrm{d}x > - \infty$, cf.~Section §5.8 in \cite{Brockwell.1991}. 
		Combining this with \eqref{eq: spectral density fGN}  and \eqref{eq: asymp spectral density}, it is easy to show that the fGN is non-deterministic.
		Thus, it fulfills the conditions of \Cref{THM: CONSISTENCY VAR PLOT}.
	\end{enumerate}
\end{Remark}

In the next section, we will apply this estimator to both finite- and infinite-variance time series.
In the latter case, we will consider subordinated Gaussian time series $Z$ as described in \eqref{eq: subordinated Gaussian}.
However, we need to transform $Z$ as described in \Cref{rem: LRD-Connections}\ref{rem: LRD-Connections pos. associated} in order to apply our estimator. 
Thus, we will compute
\begin{align}
\label{eq: non-linear transform}
Z_\nu(k) := \int_{\mathbb{R}} \mathds{1}\{ Z(k) > u \}\, \nu(\mathrm{d}u), \quad k \in \mathbb{Z},
\end{align}
where $\nu$ is a finite measure on $\mathbb{R}$. 
More precisely, we will chose $\nu = \sum_{j = 1}^J w_j \delta_{u_j}$, where $J \in \mathbb{N}$, $w_j > 0$, $j = 1, \ldots, J$, and $\delta_u$ describes the Dirac measure concentrated at $u \in \mathbb{R}$. 

Notice that \eqref{eq: non-linear transform} transforms $Z$ non-linearly. 
Even if the underlying time series $Z$ fulfills the assumptions of \Cref{THM: CONSISTENCY VAR PLOT}, it is not clear whether the same holds true for the transformed time series $Z_\nu$. 
However, the next theorem shows that $Z$ can be safely transformed as described in \eqref{eq: non-linear transform}.
The corresponding proof can be found in \Cref{appendix: trafo okay proof}.

\begin{Theorem}
	\label{THM: TRAFO IS OKAY}
	Let $Y = \{Y(k), k \in \mathbb{Z}\}$ be a stationary, non-deterministic Gaussian time series whose spectral density exists. 
	Also, assume that the autocovariance function $\gamma_Y(k) = \lgamma(k) \vert k \vert^{2d-1}$, $d \neq 0$, of $Y$ fulfils condition \eqref{eq: slowly varying convergence assumption}. 
	Further, define a subordinated Gaussian time series $Z(k) = g(Y(k))$, $k \in \mathbb{Z}$, via an even, continuous function $g: \mathbb{R} \rightarrow \mathbb{R}$ which is strictly monotonically increasing on $[0, \infty)$. \\
	For a discrete measure $\nu = \sum_{j = 1}^{J} w_j \delta_{u_j}$ where $J \in \mathbb{N}, w_j > 0$, $u_j \in \mathbb{R}$ and the time series $Z_\nu = \{ Z_\nu(k), k \in \mathbb{Z}\}$ given by \eqref{eq: non-linear transform}, it holds that $Z_\nu$ is a stationary, non-deterministic time series whose spectral density exists. 
	Furthermore, the autocovariance function of $Z_\nu$ is regularly varying and fulfills condition \eqref{eq: slowly varying convergence assumption}.
\end{Theorem}

\section{Simulation Study}

\label{sec: simulations}

In this section, we perform a simulation study in order to measure the classification performance of our estimators of $d$ empirically. 
The goal of this study is to ensure performance comparability of the variance plot estimator with the GPH estimator.
To do so, we simulated realizations from a unit-variance fGN in \Cref{sec: finite variance}. 
Additionally, we simulated infinite-variance subordinated fGNs in \Cref{sec: infinite variance}. 
Here, we used the subordinated process as described in Equation \eqref{eq: subordinated Gaussian}.
For simplicity, we speak of LRD meaning \gammaRegLRD, \SDLRD{} or \iesLRD{} (where appropriate) for the remainder of \Cref{sec: simulations}.

In both the finite and infinite-variance case, we have simulated in total $N = 12{,}000$ realizations of a (subordinated) fGN using 12 different Hurst parameters $H \in (0, 1)$, i.e.\ $1{,}000$ simulated paths for each value of $H$. The values of $H$ were chosen such that they lie equidistantly and symmetrically around the threshold values $\frac{1}{2}$ and $\frac{3}{4}$, respectively.
Furthermore, these $N$ simulations were run multiple times using varying time series length $n$.
%Note that even though we are mainly interested in LRD in this paper, we have also simulated as many non-LRD as LRD realizations.
%This helps us to avoid drawing incorrect conclusions which might arise from an unbalanced set-up.

As \Cref{sec: Estimators} demonstrated, the two considered estimators rely on running a linear regression on an observation window $[n_1, n_2]$ where $n_1, n_2 \in \mathbb{N}$ with $n_1 < n_2$.
Although these windows $[n_1, n_2]$ technically lie in different domains, we risk this slight abuse of notation for both.

For time series of length $n$ we would, in principle, choose $n_1 = \lfloor n^{\delta} \rfloor$ and $n_2 = \lceil m n^{\delta} \rceil$ for both estimators where $m > 1$ and $\delta \in (0, 1)$ are parameters.
Unfortunately, as is often the case with semi-parametric statistics, there is no obvious choice for the parameters $m$ and $\delta$.
That is why we vary both $n_1$ and $n_2$ on a grid.
Again, we are interested in comparability of the two estimators only.
We do not discuss an optimal choice of $n_1$ and $n_2$.

After a regression slope $\hat{\theta}_n$ is estimated we classify a time series as LRD if $\hat{\theta}_n > -1$ or non-LRD, otherwise.
Afterwards, we consider common classification metrics for each set of time series of length $n$.
Here, we use the metrics Accuracy, Sensitivity and Specificity.
These can be defined in terms of true/false positive/negative (TP, FP, TN, FN) as given in \cite{Tharwat.2021}:
\begin{align*}
\text{Accuracy} = \frac{
	\text{TP} + \text{TN}
}{
	\text{TP} + \text{FP} + \text{TN} + \text{FN}
}, \quad
\text{Specificity} = \frac{\text{TN}}{\text{TN} + \text{FP}}, \quad
\text{Sensitivity} = \frac{\text{TP}}{\text{TP} + \text{FN}}.	
\end{align*}

In our case, LRD is identified as ``positive". 
Here, we have chosen these metrics for no particular reason other than that they are common.
The R code and all results for this simulation study can be found in the GitHub repository \texttt{ AlbertRapp/LRD\_detection}. 

\subsection{Finite Variance Case}
\label{sec: finite variance}

In this section, we will look at $N = 12{,}000$ realizations $Y_i = \{ Y_i(k), k = 1, \dots, n \}$, $i = 1, \dots, N$, of the unit-variance fGN introduced in \Cref{def: lfsm} and \Cref{rem: lfsm}.
Moreover, the fGNs will be simulated using varying time series lengths $n \in \{ 50, 100, 200, 500 \}$ and equidistant memory parameters $0 {.}3 = H_1 < \dots < H_{12} = 0{.}7$.
These lie symmetrically around the LRD threshold $1 / 2$.

In \Cref{tbl: top metrics finite variance}, we show metrics for time series of length 200 for both GPH as well as the variance estimator.
This table compares the metrics w.r.t.~the five cutoffs that yield the highest Accuracy.
For a complete picture, we have visualized metric estimates resulting from the whole grid of cutoff values $n_1$ and $n_2$ in Figures \ref{fig: fGN metrics variance} and \ref {fig: fGN metrics GPH}. 

\begin{table}
	\centering
	\renewcommand{\arraystretch}{0.6}
	\begin{tabular}[t]{rrrrrrrrrrr}
		\toprule
		\multicolumn{5}{c}{\textbf{Variance estimator}}  & \multicolumn{1}{c}{} & \multicolumn{5}{c}{\textbf{GPH estimator}} \\
		\cmidrule(l{3pt}r{3pt}){1-5} \cmidrule(l{3pt}r{3pt}){7-11}
		$n_1$ & $n_2$ & Accuracy & Sens. & Spec. & \hspace{4mm} & $n_1$ & $n_2$ & Accuracy & Sens. & Spec.\\
		\cmidrule(l{3pt}r{3pt}){1-5} \cmidrule(l{3pt}r{3pt}){7-11}
		1 &	4 &	90.81\%	& 88.35\% &	93.27\% & & 58	& 199	& 89.08\% &	89.17\%	& 89.00\% \\
		1 & 3 &	90.79\% & 88.97\% &	92.62\% & & 59	& 199	& 89.08\% &	89.17\%	& 89.00\% \\
		1 &	2 &	90.37\%	& 89.57\% &	91.17\% & & 1	& 141	& 89.08\% &	89.17\%	& 89.00\% \\
		1 &	5 &	90.34\%	& 87.50\% &	93.18\% & & 1	& 142	& 89.08\% &	89.17\%	& 89.00\% \\
		1 &	6 &	89.73\%	& 86.37\% &	93.10\% & & 1	& 162	& 89.07\% &	89.17\%	& 88.97\% \\
		\midrule
	\end{tabular}
	\caption{Top 5 cutoffs $n_1$ and $n_2$ that yield the highest Accuracy for fractional Gaussian noise time series of length 200. Results are based on $N = 12{,}000$ realizations with Hurst parameters $0 {.}3 = H_1 < \dots < H_{12} = 0{.}7$.}
	\label{tbl: top metrics finite variance}
\end{table}

The take-away from both Table 1 and Figures \ref{fig: fGN metrics variance} and \ref{fig: fGN metrics GPH} is as follows.
Overall, the variance estimator performs best (in the sense of high accuracy) with little or no cutoff on the left, i.e. $n_1$ is close to 1.
The same thing can happen for the GPH estimator.
However, the GPH estimator can also perform well for other values of $n_1$.
%If the cutoff values are incorrectly specified, the variance estimator can achieve high specificity and comparatively low sensitivities. 
%Since LRD  is "positive", this means that the variance estimator correctly classifies many non-LRD time series at the cost of LRD ones.

In summary, the variance-based estimator performs better or at least similarly well compared to the GPH estimator.
Here, we have seen this specifically for time series of length 200.
The results for other time series lengths are similar.
And as expected, all metrics improve for both estimators as the time series length $n$ increases. 
That is why we have only showcased time series of length 200 here.
One can find the corresponding tables and figures for $n \neq 200$ in the aforementioned GitHub repository.

\begin{figure}
	\centering
	\includegraphics[width=0.9\linewidth]{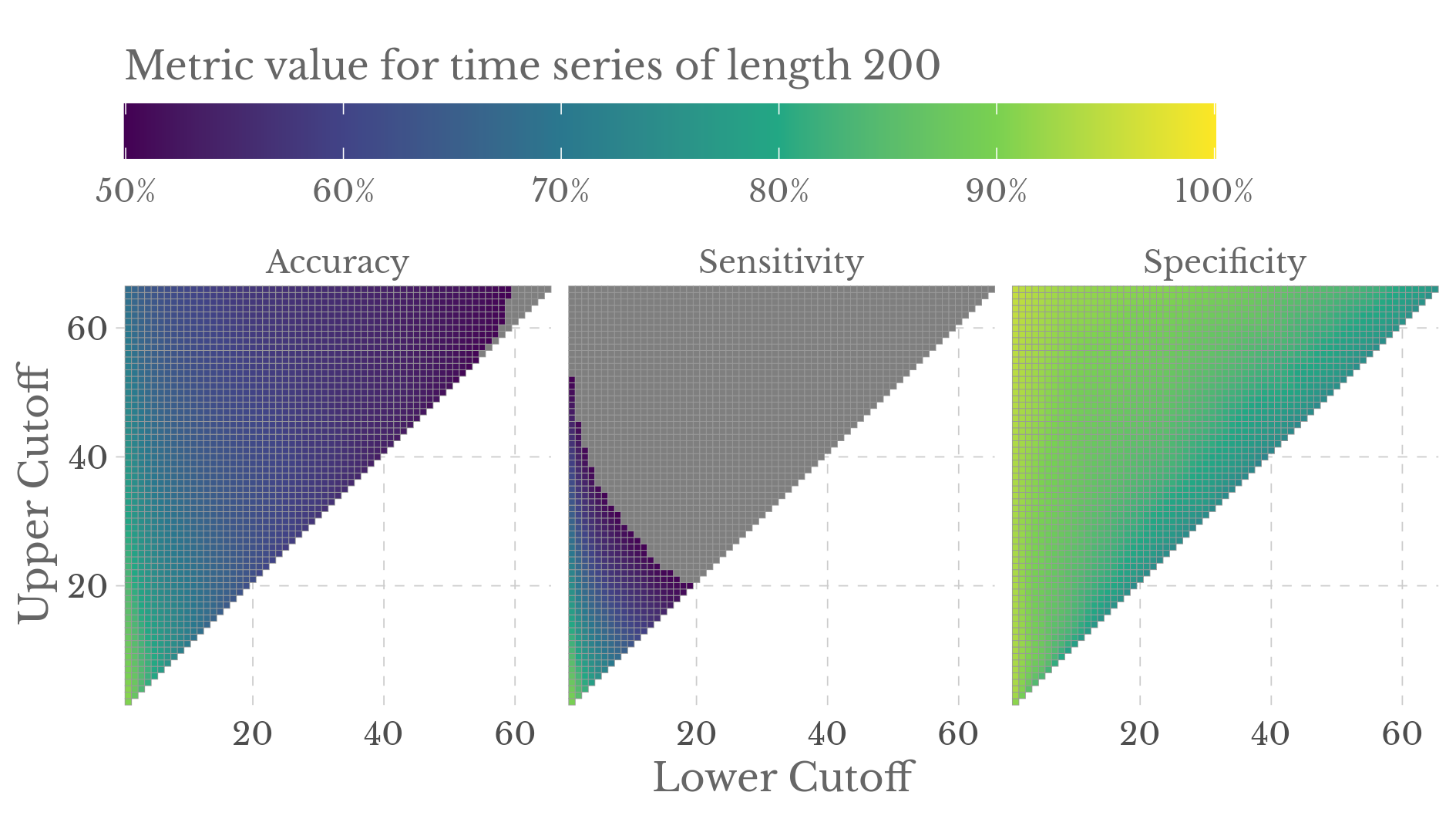}
	\caption{Evaluated metrics of variance plot estimator on observation windows $[n_1, n_2]$. Results are based on $N = 12{,}000$ fGN time series of length 200 with Hurst parameters $0 {.}3 = H_1 < \dots < H_{12} = 0{.}7$. Grey color implies that a metric was below 50\%. Preliminary analysis have shown that cutoff values larger than 60 deliver worse results. Hence, these have been left out.}
	\label{fig: fGN metrics variance}
\end{figure}

\begin{figure}
	\centering
	\includegraphics[width=0.9\linewidth]{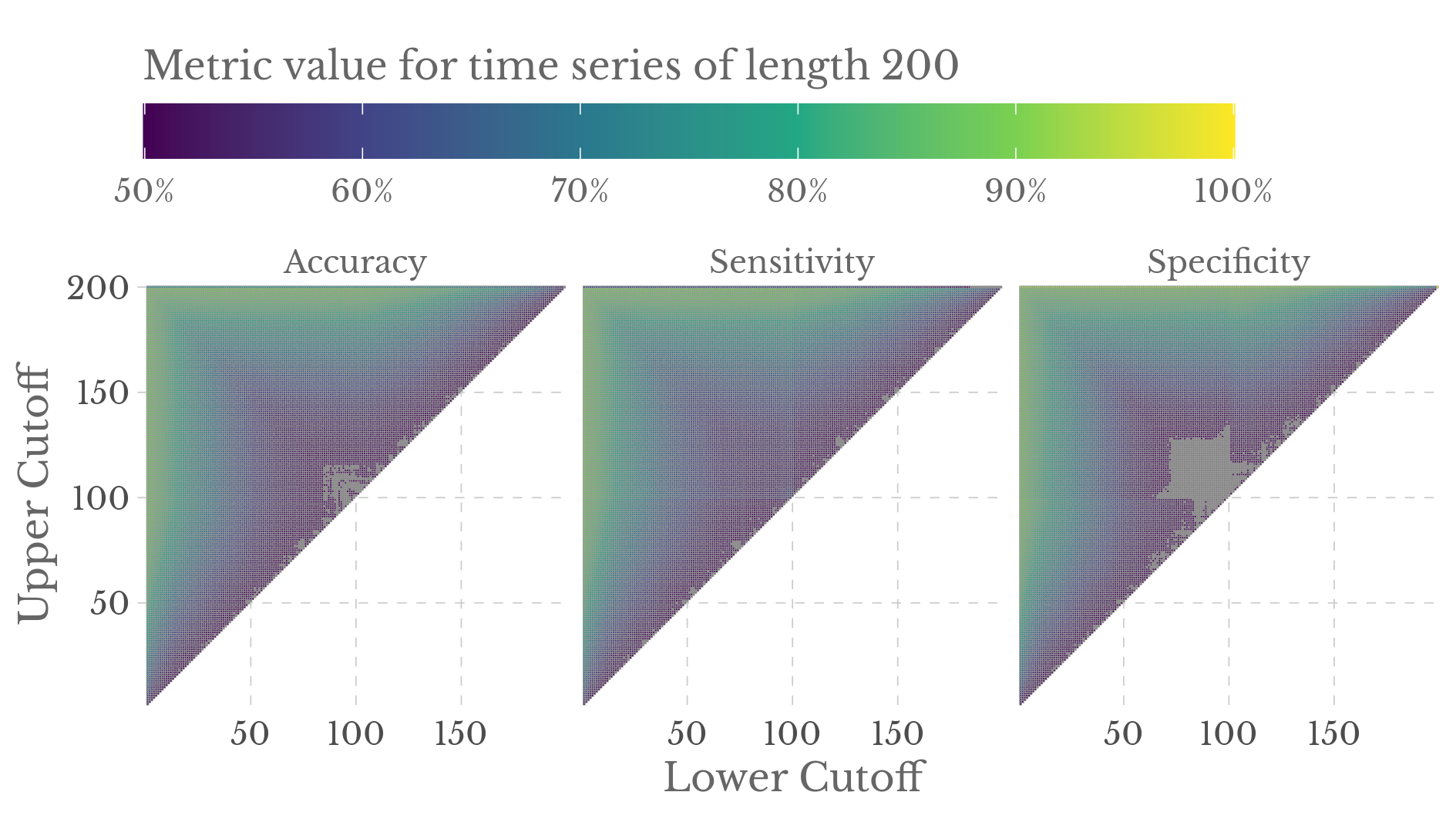}
	\caption{Evaluated metrics of GPH estimator on observation windows $[n_1, n_2]$. Results are based on $N = 12{,}000$ fGN time series of length 200 with Hurst parameters $0 {.}3 = H_1 < \dots < H_{12} = 0{.}7$. Grey color implies that a metric was below 50\%.}
	\label{fig: fGN metrics GPH}
\end{figure}

\subsection{Infinite Variance Case}
\label{sec: infinite variance}

In this section, we basically run the same analysis as in \Cref{sec: finite variance} for infinite-variance subordinated fGNs $Z(t) = \exp\{Y^2(t) / (2 \alpha)\}$ where $\alpha \leq 2$ and $Y$ is a unit-variance fGN with Hurst parameter $H$.
Clearly, LRD detection with either the variance or the GPH estimator is not tailored to infinite-variance time series of this kind.
But as was pointed out in \Cref{rem: LRD-Connections}\ref{rem: LRD-Connections pos. associated} and \Cref{thm: cov subordinated fGN}, the time series $Z$ is \iesLRD{} iff the transformed time series
\begin{align*}
Z_\nu(k) := \int_{\mathbb{R}} \mathds{1}\{ Z(k) > u \}\ \nu(\mathrm{d}u), \quad k \in \mathbb{Z},
\end{align*}
is \gammaRegLRD{} for a finite measure $\nu$.
Thus, we can transform $Z$ using a suitable measure $\nu$ and apply the variance and GPH estimators to this transformed time series. 
In the case of the variance estimator, \Cref{THM: TRAFO IS OKAY} showed that this non-linear transformation is justified and does not invalidate our main result \Cref{THM: CONSISTENCY VAR PLOT}.
Notice that it is not clear whether the same  is true for the GPH estimator.
For the sake of comparison, we apply the GPH estimator in this setting anyway.

Technically, using only a single measure $\nu$ may not suffice to detect LRD sufficiently well.
From the theoretical point of view, it is never guaranteed that we have chosen the ``right'' $\nu$.
Thus, it is still an open question how $\nu$ should be chosen in an optimal way or even if one should test multiple measures $\nu$.

However, this does not matter for comparing the variance-based estimator with the GPH estimator.
For simplicity, given a realization of a time series $X$ we have chosen $\nu$ from
\begin{align*}
\mathcal{M}_{X, \psi} = \bigg\{ \nu = \frac{1}{\psi} \sum_{k = 1}^{\psi} \delta_{X, a_k} \, \bigg\vert \, a_1, \ldots, a_\psi \in (0, 1)  \bigg\}
\end{align*}
where $\delta_{X, a_k}$ describes the Dirac measure concentrated at the $a_k$-th empirical quantile from observations in the time series $X$.
Here, we have chosen $\nu \in \mathcal{M}_{X, 100}$ where $a_1, \ldots, a_{100}$ were randomly chosen from $(0, 1)$ but are the same for every time series $X$.

Notice that the integral \eqref{eq:  iesLRD} w.r.t.\ to this measure $\nu$ will be infinite iff the integral (2.3) is infinite for at least one Dirac measure $\delta_{X, a_k}$, $k=1,\ldots,\psi$.
From a theoretical perspective, the chosen measure $\nu$ tests the finiteness of the integral \eqref{eq:  iesLRD} not only for one but for many measures.
As the results will show, our choice of $\nu$ leads to results that are sufficient for the intended comparison of estimators. 
Notice that the measures $\nu \in \mathcal{M}_{X, \psi}$ use empirical quantiles to construct suitable level sets in order to calculate the transformed process $Z_\nu$.
This is a reasonable choice to capture much of the underlying process' behavior. 
Otherwise, a bad choice of levels $a_k$ may lead to uninformative level sets.

As a consequence, the transformed process $Z_\nu$ does not change when the underlying process $Z$ is changed via a monotonic transformation.
Here, this leads to the fact that both classification procedures yield the same results for different values of $\alpha \in (0, 2]$.
Consequently, we can reduce our simulation study to $\alpha = 1$.

Now, we simulated $N = 12{,}000$ paths of our subordinated fGN using $12$ distinct Hurst parameters $0 {.}6 = H_1 < \dots < H_{12} = 0{.}9$ which are equidistantly and symmetrically distributed around the threshold value $0{.}75$.
Then, all of these realizations were transformed to a finite-variance time series using the described measure $\nu$.
Finally, the same analysis as in \Cref{sec: finite variance} was performed on these transformed time series and the results are depicted in \Cref{tbl: top metrics infinite variance} and Figures \ref{fig: subfGN metrics variance} and \ref{fig: subfGN metrics GPH}.

\begin{table}
	\centering
	\renewcommand{\arraystretch}{0.6}
	\begin{tabular}[t]{rrrrrrrrrrr}
		\toprule
		\multicolumn{5}{c}{\textbf{Variance estimator}}  & \multicolumn{1}{c}{} & \multicolumn{5}{c}{\textbf{GPH estimator}} \\
		\cmidrule(l{3pt}r{3pt}){1-5} \cmidrule(l{3pt}r{3pt}){7-11}
		$n_1$ & $n_2$ & Accuracy & Sens. & Spec. & \hspace{4mm} & $n_1$ & $n_2$ & Accuracy & Sens. & Spec.\\
		\cmidrule(l{3pt}r{3pt}){1-5} \cmidrule(l{3pt}r{3pt}){7-11}
		1 &	14 &	67.75\%	& 87.70\% &	47.80\% & & 112	& 195	& 64.36\% &	93.10\%	& 35.62\% \\
		1 & 17 &	67.74\% & 84.90\% &	50.58\% & & 5	& 88	& 64.36\% &	93.10\%	& 35.62\% \\
		1 &	15 &	67.66\%	& 86.73\% &	48.58\% & & 107	& 194	& 64.32\% &	93.20\%	& 35.43\% \\
		1 &	11 &	67.65\%	& 90.58\% &	44.72\% & & 6	& 93	& 64.32\% &	93.20\%	& 35.42\% \\
		1 &	16 &	67.65\%	& 85.65\% &	49.65\% & & 5	& 90	& 64.30\% &	93.30\%	& 35.30\% \\
		\midrule
	\end{tabular}
	\caption{Top 5 cutoffs $n_1$ and $n_2$ that yield the highest Accuracy for subordinated fractional Gaussian noise time series of length 200. Results are based on $N = 12{,}000$ realizations with Hurst parameters $0 {.}6 = H_1 < \dots < H_{12} = 0{.}9$.}
	\label{tbl: top metrics infinite variance}
\end{table}

\begin{figure}
	\centering
	\includegraphics[width=0.9\linewidth]{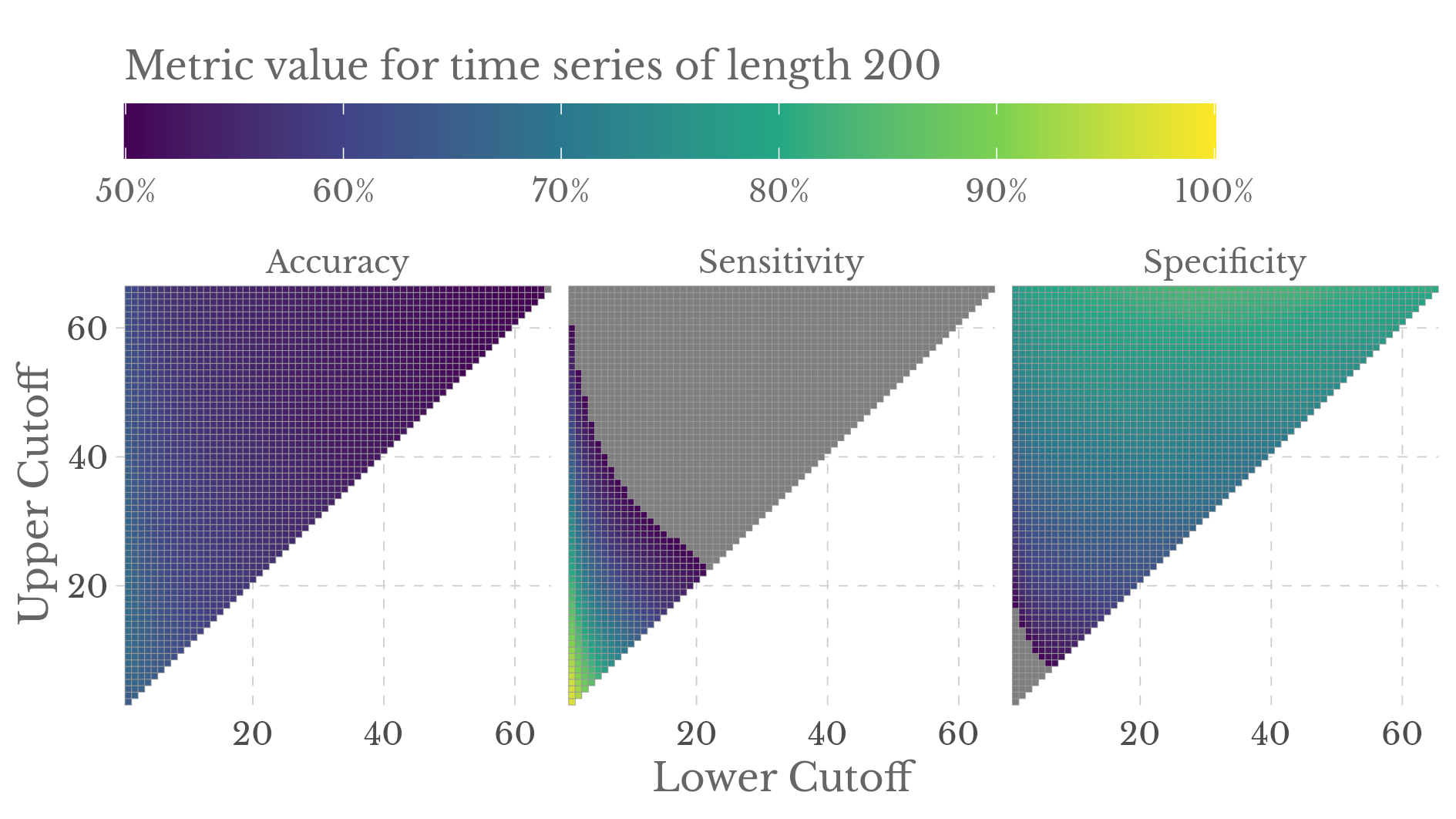}
	\caption{Evaluated metrics of variance plot estimator on observation windows $[n_1, n_2]$. Results are based on $N = 12{,}000$ subordinated fGN time series of length 200 with Hurst parameters $0 {.}6 = H_1 < \dots < H_{12} = 0{.}9$. Grey color implies that a metric was below 50\%. Preliminary analysis have shown that cutoff values larger than 60 deliver worse results. Hence, these have been left out.}
	\label{fig: subfGN metrics variance}
\end{figure}

\begin{figure}
	\centering
	\includegraphics[width=0.9\linewidth]{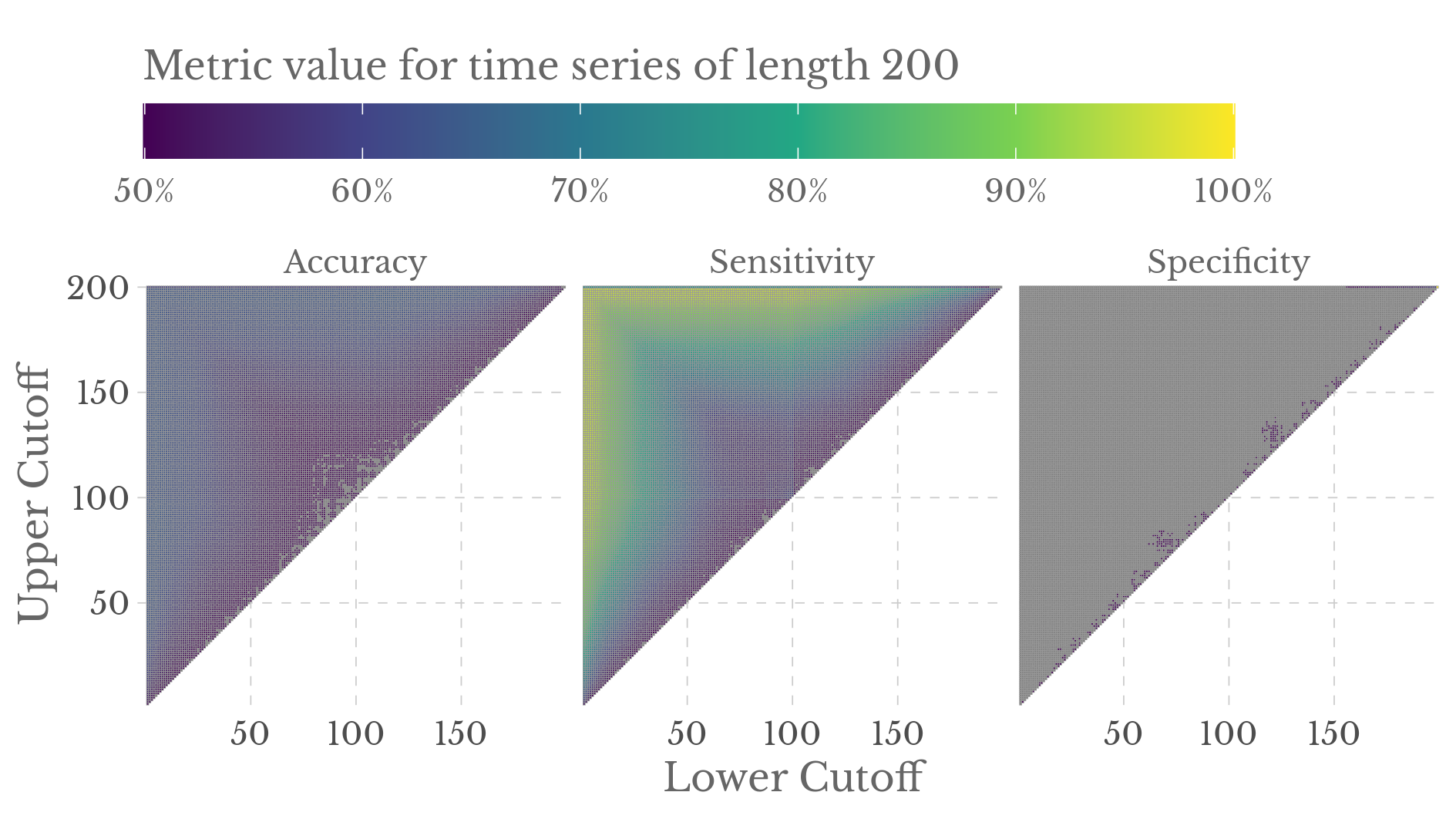}
	\caption{Evaluated metrics of GPH estimator on observation windows $[n_1, n_2]$. Results are based on $N = 12{,}000$ subordinated fGN time series of length 200 with Hurst parameters $0 {.}6 = H_1 < \dots < H_{12} = 0{.}9$. Grey color implies that a metric was below 50\%.}
	\label{fig: subfGN metrics GPH}
\end{figure}

Again, the results show that the variance estimator performs similarly well compared to the GPH estimator.
As expected, one can see that both estimators perform worse than in the finite-variance case.
Also, it appears as if the GPH estimator has a tendency towards low specificity.
In this case, this means that the variance estimator can detect LRD better.

Let us stress that, regardless of the specific choice of $\nu$, the main finding of this simulation study is that the variance-based estimator can  perform better than the GPH estimator. 
This is especially important because the latter estimator is one natural choice to detect LRD empirically whereas the former estimator is referred to ``heuristic''.

Therefore, this simulation study and the theoretical findings from \Cref{sec: consistency} demonstrate once again that time domain estimators can be more than ``heuristic". 
In fact, they can be as good or sometimes even better as spectral-domain estimators w.r.t.\ various classification metrics.
This has been explicitly demonstrated here for the variance and GPH estimators.

\section{Discussion}
\label{sec: discussion}

Let us put our work into the context of the existing literature.
\cite{Giraitis1999} consider the same variance plot estimator as we do and prove some asymptotic results.
Compared to our results, these use additional second order conditions of the behavior of the auto-covariance function.
Also, their results are tailored to Gaussian processes.
Unfortunately, these restrictions do not permit to use the variance plot estimator in the infinite-variance setting.

Let us also note that our main result (\Cref{THM: CONSISTENCY VAR PLOT}) was proven for a large class of stationary time series models.
In contrast, most existing estimators are proven for special classes like FARIMA or Gaussian processes.
However, in the infinite-variance case we cannot assume such a model class because we transform the time series such that its variance becomes finite.

Furthermore, \cite{McElroy2007} consider the so-called scans method for estimating convergence rates of a desired statistic $T$.
For example, this could be used to estimate the memory parameter $d$ by letting $T$ be the sample variance of the sample mean. 
Similar to our estimator, different values of a statistic are then  plugged into an OLS regression on a log-log scale.
We consider all blocks of a given length $n$ and combine them into one estimate for $\Var(\mean{X}{n})$ before using these estimates for regression.
However, their estimator relies on "scans".
These are nested subsamples of the observed time series $X$, e.g. $(X(1))$, $(X(1), X(2))$, $(X(1), X(2), X(3))$, etc.
Regression is then performed on $T(X(1))$, $T(X(1), X(2))$, $T(X(1), X(2), X(3)),$ etc. 
Since there exist $2^{n - 1}$ possible scans (sequences of nested subsamples) for a time series of length $n$ with no ties, their method computes many estimates of the desired convergence rates.
Subsequently, all of these estimates can be summarized into one estimate by taking the mean or the median of these estimates.
Finally, the authors of that paper comment that their procedure is computer-intensive and tolerable if only a single data set is involved but it is unsuitable for a simulation study.
Also, they propose an alternative algorithm that uses a computational shortcut.
Unfortunately, this shortcut is only valid for weakly dependent time series.
In particular, this prohibits its use in our LRD setting.

\section*{Funding}
Marco Oesting gratefully acknowledges funding by Deutsche Forschungsgemeinschaft (DFG, German Research Foundation) under Germany’s Excellence Strategy -- EXC 2075 -- 390740016.

\section*{Disclosure statement}

The authors report there are no competing interests to declare.

\section*{Data Availability Statement}
The data and code that support the findings of this study are openly available in \texttt{AlbertRapp/LRD\_detection} at \texttt{https://github.com/AlbertRapp/LRD\_detection}

\bibliography{lit}

\appendix

\section{Proof of Theorem \ref{THM: CONSISTENCY VAR PLOT}}
\label{appendix: main result}

Our proof of Theorem \ref{THM: CONSISTENCY VAR PLOT} is heavily influenced by \cite{Kim.2011}.
In that paper, the authors prove consistency of a block-bootstrap estimator for time series $X = \big\{ X(k), k \in \mathbb{Z} \big\}$ defined by
$
X(k) = \mu + \sum_{j = -\infty}^{\infty} a_j \varepsilon_{k - j}.
$
In their paper, $a_j \in \mathbb{R}$, $\mu \in \mathbb{R}$, $\sum_{j = -\infty}^{\infty} a_j^2 < \infty$ and $(\varepsilon_j)_{j \in \mathbb{Z}}$ is a sequence of identically distributed, zero mean, finite-variance, \textit{independent} random variables with  a covariance function which fulfills $
\Cov(X(0), X(k)) \sim \sigma^2 k^{-\theta}$, $k \rightarrow \infty, 
$
where $\theta \in [-1, 0)$.
Also, let us note that their proof does not actually use the independence of the innovations but relies solely on uncorrelatedness.

Thus, we first prove \Cref{THM: CONSISTENCY VAR PLOT} for linear time series with \textit{uncorrelated} innovations $\varepsilon_{j}$ in \Cref{thm: supremum convergence} which follows along the lines of the proof from \cite{Kim.2011} but makes adjustments where necessary.
At the end of this section, the proof of \Cref{THM: CONSISTENCY VAR PLOT} draws a bridge from stationary time series to linear processes with \textit{uncorrelated} innovations $\varepsilon_{j}$.

\begin{Lemma}
	\label{thm: supremum convergence}
	Let $X = \{ X(k), k \in \mathbb{N} \}$ be a time series which can be represented as
	$
	X(k) = \mu + \sum_{j = -\infty}^{\infty} a_j \varepsilon_{k - j}
	$, $k \in \mathbb{Z}$,
	where $\mu \in \mathbb{R}$, $(a_j)_{j \in \mathbb{Z}}$ is a square-summable sequence of real numbers and $(\varepsilon_j)_{j \in \mathbb{Z}}$ is a sequence of identically distributed and uncorrelated (but \textit{not} necessarily independent) random variables with $\mathbb{E}\varepsilon_1 = 0$ and $\mathbb{E}\varepsilon_1^2 < \infty$. Also, assume that the autocovariance function $\gamma(n) = L_{\gamma}(n)n^{2d-1}$, $d \neq 0$, of $X$ fulfills $L_{\gamma}(n) \rightarrow \cgamma \neq 0$ as $n \rightarrow \infty$.\\
	Furthermore, define estimates of $\Var(\mean{X}{l}) = \lvar(l) l^{\theta}$, $\theta \in (-2, 0)$, for $l = 1, \ldots, n$ via
	$$
	S_l^2 := \frac{1}{n - l + 1} \sum_{k = 1}^{n - l + 1} (\mean{B}{k,l} - \muhat)^2,
	$$
	where $\mean{B}{k,l}$, $k = 1, \dots, n - l + 1$, denotes the mean of the block $(X(k), \dots, X(k + l - 1))$ and $\muhat$ is the sample mean of all block means $\mean{B}{k,l}$ of length $l$.
	
	If $n_1 = n^{\delta}$ and $n_2 = m n_1$ where $m > 1$ and $0 < \delta <  \min \Big\{ \frac{2 \vert \theta \vert}{4 \vert \theta \vert + 1} ,  \frac{\vert \theta \vert}{\vert \theta \vert + (|\theta|-1)_+ + 1}\Big\}$, it holds that
	$$
	\sup_{k = 1, \dots, N} \vert  \yk - (\log \cvar - \vert \theta \vert \xk) \big\vert \stackrel{\text{P}}{\rightarrow} 0
	$$
	as $n \rightarrow \infty$ where $N = n_2 - n_1 + 1$, $\xk = \log(n_1 + k - 1)$, $\yk = \log S^2_{n_1 + k - 1}$, $k = 1, \dots, N$ and $\cvar = \lim_{n \rightarrow \infty} n^{\vert \theta \vert} \Var(\mean{X}{n}) = \cgamma v(d)$ is a positive constant given by \Cref{thm: variance-behaviour} and \Cref{rem: behaviour of variance} \ref{item: cvar and cgamma}.
\end{Lemma}

\begin{proof}
	Let $\varepsilon > 0$ and $\mathcal{N}= \{ n_1, \ldots, n_2 \}$. We would like to show that
	$$
	\mathbb{P} \big(
	\sup_{k = 1, \dots, N} \big\vert \yk - (\log \cvar - \vert \theta \vert \xk) \big\vert > \varepsilon
	\big) 
	=
	\mathbb{P} \big(
	\sup_{l \in \mathcal{N}} \big\vert \log\big( l^{\vert \theta \vert} S^{2}_l  \big) - \log \cvar \big\vert > \varepsilon
	\big)
	$$
	converges to zero as $n \rightarrow \infty$ for any $\varepsilon > 0$. 
	Instead of considering the supremum of the difference of logarithms, we can also consider the difference of the original quantities, i.e.\ we need to show
	$
	\mathbb{P} \big(
	\sup_{l \in \mathcal{N}} \big\vert l^{\vert \theta \vert} S^2_{l} - \cvar  \big\vert > \varepsilon
	\big) 
	\rightarrow
	0. 
	$
	Notice that in this proof we always use $l \in \mathcal{N}$ and $\frac{l}{n} \rightarrow 0$ as $n \rightarrow \infty$~.
	
	Thus, let us assume w.l.o.g~$\mu = 0$ and compute
	\begin{align*}
	&\mathbb{P} \Big(
	\sup_{l \in \mathcal{N}} \big\vert l^{\vert \theta \vert} S^2_{l} - \cvar  \big\vert > \varepsilon
	\Big) 
	=
	\mathbb{P} \bigg(
	\sup_{l \in \mathcal{N}} \bigg\vert  
	\frac{l^{\vert \theta \vert}}{n - l + 1} \sum_{i = 1}^{n - l + 1} \mean{B}{i,l}^2
	-
	l^{\vert \theta \vert} \hat{\mu}_{n, l}^2
	-
	\cvar
	\bigg\vert > \varepsilon
	\bigg) \\
	&=
	\mathbb{P} \Big(
	\sup_{l \in \mathcal{N}} \big\vert  
	\sumBlockSquares
	-
	l^{\vert \theta \vert} \hat{\mu}_{n, l}^2
	-
	\cvar
	\big\vert > \varepsilon
	\Big) 
	\leq
	\mathbb{P}\Big( 
	\sup_{l \in \mathcal{N}} \big\vert  
	\sumBlockSquares
	-
	\cvar
	\big\vert
	> 
	\frac{\varepsilon}{2}
	\Big)
	+
	\mathbb{P}\bigg( 
	\sup_{l \in \mathcal{N}} 
	l^{\vert \theta \vert} \hat{\mu}_{n, l}^2
	>
	\frac{\varepsilon}{2}
	\bigg),
	\end{align*}
	where $\sumBlockSquares = \frac{l^{\vert \theta \vert}}{n - l + 1} \sum_{i = 1}^{n - l + 1} \mean{B}{i,l}^2$.
	Analogously to \cite{Kim.2011}, we obtain
	\begin{align*}
	l^{\vert \theta \vert}\mathbb{E}[ \hat{\mu}_{n, l}^2] 
	&= 
	l^{\vert \theta \vert}\mathrm{Var}( \hat{\mu}_{n, l}) \\
	&\leq 
	\frac{l^{\vert \theta \vert}}{(n - l + 1)^2} 
	\Var \bigg( 
	n\mean{X}{n} 
	- 
	\sum_{k = 1}^{k} \Big( 1 - \frac{k}{l} \Big)X(k)
	-
	\sum_{k = 1}^{k} \Big( 1 - \frac{k}{l} \Big)X(n - k + 1)
	\bigg) \\
	&\leq
	\frac{l^{\vert \theta \vert}}{(n - l + 1)^2} \bigg( 
	3 \Var(n\mean{X}{n}) + 6 \Var \bigg( \sum_{k = 1}^{k} \Big( 1 - \frac{k}{l} \Big)X(k) \bigg)
	\bigg) 
	\end{align*}
	by using that $\Var(X) + \Var(Y) - 2 \Cov(X, Y) \geq 0$ for any random variables $X, Y \in \mathcal{L}^2$.\\
	Notice that
	$
	(n-l+1)^{-2} \mathrm{Var}(n \mean{X}{n}) \sim \mathrm{Var}(\mean{X}{n}) \sim c_v n^{-|\theta|} 
	$
	and 
	\begin{align*}
	\frac{1}{(n-l+1)^2} &\mathrm{Var}\left(\sum_{k=1}^l  \bigg(1-\frac{k}{l}\bigg)  X(k)\right) 
	\leq 
	\frac{l}{(n-l+1)^2}  \sum_{k=0}^l |\gamma(k)| = \mathcal{O}\left(\frac{l^{2- (|\theta| \wedge 1)}}{n^2}\right).
	\end{align*}
	Thus, there is a constant $C > 0$ such that $l^{\vert \theta \vert}\mathbb{E}[ \hat{\mu}_{n, l}^2] \leq C \big( \frac l n \big)^{|\theta|} \cdot l^{(|\theta|-1)_+}$.\\
	Now, applying Markov's inequality gives us
	\begin{align}
	\mathbb{P}\bigg( 
	\sup_{l \in \mathcal{N}} 
	l^{\vert \theta \vert} \hat{\mu}_{n, l}^2
	>
	\frac{\varepsilon}{2}
	\bigg)
	\leq
	C_1 \sum_{l = n_1}^{n_2} l^{\vert \theta \vert} \mathbb{E}[ \hat{\mu}_{n, l}^2] 
	\leq
	%			C_2  \sum_{l = n_1}^{n_2}  \frac{l^{\vert \theta \vert + (|\theta|-1)_+}}{n^{\vert \theta \vert}}
	%			\leq
	C_2 (n_2 - n_1 + 1) \frac{n_2^{\vert \theta \vert + (|\theta|-1)_+}}{n^{\vert \theta \vert}},
	\label{eq: bound after Markov}
	\end{align}
	where $C_1$ and $C_2$ are suitable positive constants.

	Clearly, the RHS of Inequality \eqref{eq: bound after Markov} converges to zero for $n_1 = n^{\delta}$, $n_2 = mn^\delta$ and $\delta < \frac{\vert \theta \vert}{\vert \theta \vert + (|\theta|-1)_+ + 1}$.
	Thus, it remains to show that $\mathbb{P}(\sup_l \vert \sumBlockSquares - \cvar \vert > \varepsilon / 2)$ converges to zero.
	In alignment with the proof of Lemma A.1(b) in \cite{Kim.2011}, let $b > 1$ and define a modified version of our time series $X$ via bounded and centered innovations, i.e.\
	$
	\varepsilon_{j, b} 
	= 
	\varepsilon_j \mathds{1}\{ \vert \varepsilon_j \vert \leq b\}
	- 
	\mathbb{E}[\varepsilon_j \mathds{1}\{ \vert \varepsilon_j \vert \leq b\}]
	$ and $
	X_b(k)
	=
	\sum_{j = -\infty}^{\infty}a_j \varepsilon_{k-j, b}. 
	$
	Consequently, we can define corresponding quantities for $\mean{B}{i, l}$ and $\sumBlockSquares$ as
	\begin{align*}
	\mean{B}{i, l, b} = \frac{1}{l} \sum_{k = i + 1}^{i + l} X_b(k)
	\quad
	\text{ and }
	\quad
	\sumBlockSquaresMod 
	= 
	\frac{l^{\vert \theta \vert}}{n - l + 1} \sum_{k = 1}^{n - l + 1} \mean{B}{k, l, b}^2.
	\end{align*}
	Also, notice that $1 \geq c_b := \mathbb{E}\varepsilon_{j, b}^2/\mathbb{E}\varepsilon_{j}^2 \rightarrow 1$ as $b \rightarrow\infty$. 
	Now, it holds that
	\begin{align*}
	&\mathbb{P}\Big( 
	\sup_{l \in \mathcal{N}} \big\vert  
	\sumBlockSquares
	-
	\cvar
	\big\vert
	> 
	\frac{\varepsilon}{2}
	\Big) \\
	&\leq
	\mathbb{P}\Big( 
	\sup_{l \in \mathcal{N}} \big\vert  
	\sumBlockSquares
	-
	\sumBlockSquaresMod
	\big\vert
	> 
	\frac{\varepsilon}{6}
	\Big) 
	+ % end first summand
	\mathbb{P}\Big( 
	\sup_{l \in \mathcal{N}} \big\vert  
	\sumBlockSquaresMod
	-
	c_b\cvar
	\big\vert
	> 
	\frac{\varepsilon}{6}
	\Big) + % end second summand
	\mathbb{P}\Big( 
	\sup_{l \in \mathcal{N}} \big\vert  
	c_b\cvar
	-
	\cvar
	\big\vert
	> 
	\frac{\varepsilon}{6}
	\Big) \\
	% end third summand
	&=: \text{(I)} + \text{(II)} + \text{(III)}.
	\end{align*}
	Next, let us show that all summands (I), (II) and (III) converge to zero as $n \rightarrow \infty$.
	First, notice that (III) converges to zero as $b \rightarrow \infty$ by the (non-random) convergence $c_b \rightarrow 1$ as $b \rightarrow \infty$.
	Therefore, ensuring the desired asymptotic behavior of (III) is only a matter of replacing $b$ by a constructed sequence $b_n$ that depends on $n$ such that $b_n \rightarrow \infty$ as $n \rightarrow \infty$.
	Additionally, choosing the sequence $b_n$ such that it diverges fast enough in the sense that 
	$
	n^\delta \sqrt{\mathbb{E}[\varepsilon_{0}^2 \mathds{1} \{ \vert \varepsilon_{0} \vert > b_n \}]} \rightarrow 0
	$
	as $n \rightarrow \infty$ ensures that also (I) converges to zero as $n \rightarrow \infty$.
	This follows from the fact that there exists a constant $C > 0$ such that 
	\begin{align*}
	\mathbb{P}\Big( 
	\sup_{l \in \mathcal{N}} \big\vert  
	\sumBlockSquares
	-
	\sumBlockSquaresMod
	\big\vert
	> 
	\frac{\varepsilon}{6}
	\Big)
	&\leq \sum_{l=n_1}^{n_2} \mathbb{P}\Big( 
	\big\vert  
	\sumBlockSquares
	-
	\sumBlockSquaresMod
	\big\vert
	> 
	\frac{\varepsilon}{6}
	\Big) \leq	
	\frac{6C(n_2 - n_1 +1)}{\varepsilon} \sqrt{\mathbb{E}[\varepsilon_{0}^2 \mathds{1} \{ \vert \varepsilon_{0} \vert > b_n \}]}.
	\end{align*}
	In the latter inequality, we have used Markov's inequality and 
	$
	\sup_{n \geq 1} \mathbb{E} \vert \sumBlockSquares - \sumBlockSquaresMod \vert
	\leq
	C \sqrt{\mathbb{E}[\varepsilon_{0}^2 \mathds{1} \{ \vert \varepsilon_{0} \vert > b \}]}
	$
	which was demonstrated in the proof of Lemma A.1(b) in \cite{Kim.2011}.
	
	In summary, we have proven that there exists a sequence $b_n$ such that (I) and (III) converge to zero as $n \rightarrow \infty$.
	So, let us complete the proof by showing the convergence of (II).
	As we will see, this is independent of the choice of $b_n$.
	To do so, compute
	\begin{align*}
	\mathbb{P}\Big( 
	\sup_{l \in \mathcal{N}} \big\vert  
	\sumBlockSquaresMod
	-
	c_b\cvar
	\big\vert
	> 
	\frac{\varepsilon}{6}
	\Big)
	&\leq %inequality
	\mathbb{P}\Big( 
	\sup_{l \in \mathcal{N}} \big\vert  
	\sumBlockSquaresMod
	-
	c_b l^{\vert \theta \vert} \Var(\mean{X}{l})
	\big\vert
	> 
	\frac{\varepsilon}{12}
	\Big) \\
	&+ % summand 1/2
	\mathbb{P}\Big( 
	\sup_{l \in \mathcal{N}} \big\vert  
	c_b l^{\vert \theta \vert} \Var(\mean{X}{l})
	-
	c_b\cvar
	\big\vert
	> 
	\frac{\varepsilon}{12}
	\Big).
	\end{align*}
	Clearly, the latter summand converges to zero because $l^{\vert \theta \vert}\Var(\mean{X}{l}) \rightarrow \cvar$ as $l \rightarrow \infty$ by assumption.
	Moreover, by Chebyshev's inequality and $c_b l^{\vert \theta \vert} \Var(\mean{X}{l}) = \mathbb{E}\sumBlockSquaresMod$ we get that
	\begin{align*}
	\mathbb{P}\Big( 
	\sup_{l \in \mathcal{N}} \big\vert  
	\sumBlockSquaresMod
	-
	c_b l^{\vert \theta \vert} \Var(\mean{X}{l})
	\big\vert
	> 
	\frac{\varepsilon}{12}
	\Big)
	\leq %inequality
	\frac{144}{\varepsilon^2}
	\sum_{l = n_1}^{n_2} \Var(\sumBlockSquaresMod).
	\end{align*}
	Recall that $\Var(\mean{X}{l}) = \lvar(l) l^{\theta}$. Now, from the proof of Lemma A.1(b) in \cite{Kim.2011} we know that there exists a constant $C > 0$ such that for any $\xi \in (0, 1)$ 
	\begin{align*}
	\Var(\sumBlockSquaresMod) 
	\leq
	c_b^2 \bigg( 4\xi
	\lvar^{4}(l)
	+ 
	4 \max_{n\xi \leq k \leq n} \Big( l^{\vert \theta \vert}\Cov(\mean{B}{1, l}, \mean{B}{1 + k, l})\Big)^2
	+
	\frac{C}{n - l + 1}
	\bigg).
	\end{align*}
	Consequently, for $n$ large enough there is another constant $C$ such that
	\begin{align}
	\sum_{l = n_1}^{n_2} \Var(\sumBlockSquaresMod) \leq %inquality
	C \bigg( 
	(n_2 - n_1 + 1) \xi
	+ % summand
	n_2^{2\vert \theta \vert}
	\sum_{l = n_1}^{n_2} \max_{n\xi \leq k \leq n} \Cov(\mean{B}{1, l}, \mean{B}{1 + k, l})^2
	+ % summand
	\frac{n_2 - n_1 + 1}{n - n_2 + 1}
	\bigg).
	\label{eq: bound sum of variances}
	\end{align}
	Obviously, for $n_1 = n^\delta$ and $n_2 = m n^\delta$ the last fraction in \eqref{eq: bound sum of variances} converges to zero.
	Similarly, the quantity $(n_2 - n_1)\xi$ goes to zero if $\xi = \xi_n$ is chosen such that $\xi_n$ depends on $n$ and converges to zero fast enough.
	Consequently, we only have to consider how the covariance term in \eqref{eq: bound sum of variances} behaves as $n$ goes to infinity.\\
	Recall that $\gamma(k) := \Cov(X(0), X(k)) = L_\gamma(k) \vert k \vert^{-\vert \theta \vert}$  where $L_\gamma$ is a slowly varying function s.t. $L_\gamma(k) \rightarrow \cgamma \neq 0$ as $k \rightarrow \infty$.
	We use this in order to establish that
	\begin{align*}
	\vert \Cov(\mean{B}{1, l}, \mean{B}{1 + k, l}) \vert 
	\leq %inequality
	\frac{1}{l} \sum_{i = -(l - 1)}^{l - 1}
	\bigg( 
	1 - \frac{\vert i \vert}{l}
	\bigg)
	\vert
	\gamma(k + i)
	\vert 
	\leq % asymptotic equivalence
	\frac{2 \vert \cgamma \vert}{l}
	\sum_{i = -(l - 1)}^{l - 1}
	\bigg( 
	1 - \frac{\vert i \vert}{l}
	\bigg)
	\vert
	k + i
	\vert^{-\vert \theta \vert}.
	\end{align*} 
	Next, we need to choose $\delta$ and $\xi_n$ appropriately such that $(n_2 - n_1 + 1) \xi_n \rightarrow 0$ as $n \rightarrow \infty$ and
	\begin{align}
	n^\delta \leq l \leq mn^{\delta} \leq n \xi_n.
	\label{eq: xi_conditios}
	\end{align}
	An appropriate choice will be discussed at the end of this proof.
	Assuming \eqref{eq: xi_conditios}, for every $k \geq n \xi_n$, it holds
	\begin{align*}
	\frac{2 \vert \cgamma \vert}{l}
	\sum_{i = -(l - 1)}^{l - 1}
	\bigg( 
	1 - \frac{\vert i \vert}{l}
	\bigg)
	&\vert
	k + i
	\vert^{-\vert \theta \vert}
	\leq % inequality
	2\vert \cgamma \vert (n\xi_n - l + 1)^{-\vert \theta \vert} 
	\frac{1}{l} \sum_{i = -(l - 1)}^{l - 1}
	\bigg( 
	1 - \frac{\vert i \vert}{l}
	\bigg) \\
	&\sim % asymptotic equality
	2 \vert \cgamma \vert (n\xi_n - l + 1)^{-\vert \theta \vert}.
	\end{align*}
	Thus, for the covariance term in \eqref{eq: bound sum of variances} we get
	\begin{align}
	n_2^{2\vert \theta \vert}
	\sum_{l = n_1}^{n_2} \max_{n\xi \leq k \leq n} \Cov(\mean{B}{1, l}, \mean{B}{1 + k, l})^2
	\leq %inequality
	8 n_2^{2\vert \theta \vert} (n_2 - n_1 + 1) 
	\cgamma^2
	(n \xi_n - n_2 + 1)^{-2\vert \theta \vert}
	\label{eq: bound sum_max_covariance}
	\end{align}
	for $n$ large enough.
	Plugging in $n_1 = n^{\delta}$ and $n_2 = m n^{\delta}$, it is easy to see that the RHS of Inequality \eqref{eq: bound sum_max_covariance} can be rewritten as
	\begin{align}
	8m^{2 \vert \theta \vert} \cgamma^2(m - 1) \bigg( \frac{n^{\delta(1 + 1/(2\vert \theta \vert)) }}{n \xi_n - m n^{\delta} + 1} \bigg)^{2 \vert \theta \vert} 
	+ % summand
	8m^{2 \vert \theta \vert} \cgamma^2 \bigg( \frac{n^{\delta}}{n \xi_n - m n^{\delta} + 1} \bigg)^{2 \vert \theta \vert}.
	\label{eq: sum of convergence}
	\end{align}
	Finally, put $\xi_n = n^{-(\delta + x)}$ where $x > 0$ can be chosen such that the summands in equality \eqref{eq: sum of convergence} converge to zero. 
	After rewriting
	\begin{align*}
	\frac{n^{\delta(1 + 1/(2\vert \theta \vert)) }}{n \xi_n - m n^{\delta}} 
	= %equality
	\frac{n^{2\delta - 1 + x + \delta / (2\vert \theta \vert)}}{1 - m n^{2\delta - 1 + x}}
	\end{align*}
	it becomes clear that $\delta$ and $x$ need to be chosen such that both
	$
	2\delta - 1 + x  < 0$ and $
	2\delta - 1 + x  + \frac{\delta}{2\vert \theta \vert} < 0 
	$
	are fulfilled.
	Notice that in order for all of these inequalities to be well-defined we need to ensure that $\delta < \frac{2 \vert \theta \vert}{4 \vert \theta \vert + 1}$ but given this choice it is easy to see that $\xi_n$ fulfils condition \eqref{eq: xi_conditios}. % which completes the proof.
\end{proof}

Before we can prove \Cref{THM: CONSISTENCY VAR PLOT}, we need a technical lemma first.

\begin{Lemma}
	\label{thm: auxillary bounds}
	Let $n_1 < n_2$ be two positive integers dependent on $n \in \mathbb{N}$ such that $n_2 - n_1 \rightarrow \infty$ and $\frac{n_2}{n_1} \rightarrow m > 1$ as $n \rightarrow \infty$.
	Further, define $N = n_2 - n_1 + 1$ and $\xk = \log(n_1 + k - 1)$ for $k = 1, \ldots, N$. 
	Then, for $n$ large enough there exist constants $C_1, C_2 > 0$ such that
	\begin{align*}
	\sup_{k = 1, \dots, N} \vert \xk - \mean{x}{N}\vert \leq C_1
	\quad 
	\text{ and }
	\quad
	\sum_{k = 1}^{N}(\xk -\mean{x}{N})^2 \geq C_2 N.
	\end{align*}

	\begin{proof}
		The first inequality follows from
		$\vert \xk - \mean{x}{N}\vert \leq \vert x_{1,n} - x_{N,n} \vert = \big\vert \log\big(\frac{n_1}{n_2}\big)\big\vert$ which converges to $\vert \log m \vert$ as $n \rightarrow \infty$.
		Next, let us compute
		\begin{align*}
		&\frac 1 N \sum_{k = 1}^{N}(\xk -\mean{x}{N})^2 
		={} \frac 1 N \sum_{k = 1}^{N}\big(\xk - \log(n_2) -[\mean{x}{N} - \log(n_2)]\big)^2 \\
		&={} \frac 1 {N} \sum_{k = 1}^{N} \log\left(\frac{n_1+k-1}{n_2}\right)^2 - \frac 1 {N} \sum_{k = 1}^{N} \log\left(\frac{n_1+k-1}{n_2}\right) \cdot \frac 1 {N} \sum_{k = 1}^{N} \log\left(\frac{n_1+k-1}{n_2}\right).
		\end{align*}
		By integrability of the functions $x \mapsto \log(x)$ and $x \mapsto x^2$, this implies
		\begin{align*}
		&\frac 1 N \sum_{k = 1}^{N}(\xk -\mean{x}{N})^2 \sim{} \frac{n_2}{N}\int_{n_1/n_2}^1 \log(x)^2 \, \mathrm{d}x - \left[\frac{n_2}{N} \int_{n_1/n_2}^1 \log(x) \, \mathrm{d}x\right]^2 \\
		\sim{}& \frac{m}{m-1}\int_{1/m}^1 \log(x)^2 \, \mathrm{d}x - \left[\frac{m}{m-1} \int_{1/m}^1 \log(x) \, \mathrm{d}x\right]^2 
		={} \Var(\log(U))
		\end{align*}  
		for large $n$, where $U$ denotes a uniform random variable on $[1/m,1]$. Thus, 
		$ \sum_{k = 1}^{N}(\xk -\mean{x}{N})^2 \geq{} C_2N $
		for some constant $C_2>0$.
	\end{proof}
\end{Lemma}	

\begin{proof}[Proof of Theorem \ref{THM: CONSISTENCY VAR PLOT}:]
	Let $X$ be any stationary time series $X$ whose spectral distribution is absolutely continuous with a spectral density denoted by $f$. From Theorem 2 in Chapter IV§7 of \cite{Gihman.1974} it follows that $X$ admits representation
	$
	X(k) = \sum_{j = 0}^{\infty} a_j \varepsilon_{k - j}
	$
	where $a_j \in \mathbb{R}$, $\sum_{j = -\infty}^{\infty} a_j^2 < \infty$ and $(\varepsilon_j)_{j \in \mathbb{Z}}$ is a sequence of uncorrelated random variables with common mean zero and common variance $\sigma^2 > 0$ iff 
	\begin{align} 
	\label{eq: non-deterministic condition}
	\int_{-\pi}^{\pi} \log f(\lambda)\ \mathrm{d}\lambda > -\infty.
	\end{align}
	Notice that condition \eqref{eq: non-deterministic condition} is fulfilled in our setting because our time series is non-deterministic, cf.~Remark 1 in §5.8 of \cite{Brockwell.1991}.
	Next, let us prove that the innovations $\varepsilon_{j}$, $j \in \mathbb{Z}$, are identically distributed. 
	From Wold's decomposition, c.f.~Thm.~5.7.1 in \cite{Brockwell.1991}, it follows that an innovation $\varepsilon_{j}$ at time point $j \in \mathbb{Z}$ is nothing but the best linear prediction of $X(j)$ by the observed past $\overline{\text{span}} \{X(s), s \leq j - 1 \}$, i.e.
	$
	\varepsilon_{j} = X(j) - \mathcal{P}_{\mathcal{M}_{j - 1}} X(j),
	$
	where $\mathcal{P}_{\mathcal{M}_{j - 1}}$ is the projection operator on $\mathcal{M}_{j - 1} := \overline{\text{span}} \{X(s), s \leq j-1 \}$.  \\
	Notice that this projection can be expressed as a linear combination of $X(s)$, $s \leq j - 1$.
	Therefore, $\varepsilon_{j} = g\big((X(j), X(j - 1), \ldots))\big)$ where $g$ is a measurable function mapping a sequence of $\mathcal{L}^2$-variables to an $\mathcal{L}^2$-variable. By stationarity of $X$ we get that $\varepsilon_{j}$, $j \in \mathbb{Z}$ are identically distributed.   
	Recall that our estimator $\thetaest$ defined in Equation \eqref{eq: slope estimator variance} writes
	\begin{align}
	\thetaest = \frac{\sum_{k = 1}^{N}(\xk - \mean{x}{N})(\yk - \mean{y}{N})}{\sum_{k = 1}^{N}(\xk - \mean{x}{N})^2}
	\end{align}
	with $N = n_2 - n_1 + 1$, $\xk = \log(n_1 + k - 1)$, $\yk = \log S^2_{n_1 + k - 1}$, $k = 1, \dots, N$ where $\mean{x}{N}$ and $\mean{y}{N}$ represent the mean of $\xk$ and $\yk$, $k = 1, \dots, N$, respectively.
	Thus,
	\begin{align*}
	\vert \thetaest - \theta \vert 
	=
	\bigg\vert
	\bigg( 
	\sum_{k = 1}^N (\xk - \mean{x}{N})^2
	\bigg)^{-1}
	\sum_{k = 1}^{N} (\xk - \mean{x}{N}) \big(\yk - \mean{y}{N} - (\xk - \mean{x}{N})\theta\big)
	\bigg\vert.
	\end{align*}
	By triangle inequality and 
	\begin{align*}
	\big\vert\yk - \mean{y}{N} - (\xk - \mean{x}{N})\theta\big\vert 
	\leq
	\big\vert \yk - (\log \cvar - \vert \theta \vert \xk) \big\vert
	+
	\frac{1}{N} \sum_{j = 1}^{N} \big\vert y_{j, n} - (\log \cvar - \vert \theta \vert \xk) \big\vert
	\end{align*}
	it follows that 
	\begin{align*}
	\vert \thetaest - \theta \vert 
	&\leq
	\Big( 
	\sum_{k = 1}^{N}(\xk -\mean{x}{N})^2
	\Big)^{-1}
	\Big( 
	\sum_{k = 1}^{N}
	\vert \xk - \mean{X}{N} (\yk - (\log \cvar - \vert \theta \vert \xk)) \vert
	\\
	&\hspace{4cm} +
	\sum_{k = 1}^{N}
	\vert \xk - \mean{X}{N}\vert
	\frac{1}{N} \sum_{j = 1}^{N} \vert y_{j, n} - (\log \cvar - \vert \theta \vert x_{j, n}) \vert
	\Big)
	\\ 
	&\leq
	2 
	\Big( 
	\sup_{k = 1, \dots, N} \vert \xk - \mean{x}{N}\vert 
	\Big)
	\Big( 
	\sum_{k = 1}^{N}(\xk -\mean{x}{N})^2
	\Big)^{-1}
	\sum_{k = 1}^{N} \big\vert \yk - (\log \cvar - \vert \theta \vert \xk) \big\vert.
	\end{align*}
	Therefore, we can use \Cref{thm: auxillary bounds} to show that there exists a constant $C \in (0, \infty)$ such that
	\begin{align}
	\label{eq: bound difference theta and theta hat}
	\vert \thetaest - \theta \vert \leq C \sup_{k = 1, \dots, N} \big\vert  \yk - (\log \cvar - \vert \theta \vert \xk) \big\vert.
	\end{align}
	By \Cref{thm: supremum convergence}, the supremum in Inequality \eqref{eq: bound difference theta and theta hat} converges to 0 in probability as $n \rightarrow \infty$. 
\end{proof}

\section{Proof of Theorem \ref{THM: TRAFO IS OKAY}}
\label{appendix: trafo okay proof}

As is common with subordinated Gaussian time series, our proofs will rely on Hermite polynomials. These can be found in many textbooks like Chapter 6.3 in \cite{Samorodnitsky.2016}. 

Let $\varphi$ be the density of a standard Gaussian random variable and define 
\begin{align*}
H_n(x) &= (-1)^n e^{x^2/2} \frac{\mathrm{d}^n}{\mathrm{d}x^n} e^{-x^2/2}, \quad x \in \mathbb{R}, \ n \in \mathbb{N}, \\
\mathcal{L}^2(\varphi) &= \bigg\{ f: \mathbb{R} \rightarrow \mathbb{R} \, \bigg\vert \, \int_{\mathbb{R}} f^2(x) \varphi(x) \mathrm{d}x < \infty \bigg\},\\
\langle f, g \rangle_{\mathcal{L}^2(\varphi)} &= \int_{\mathbb{R}} f(x)g(x) \varphi(x) \mathrm{d}x, \quad f, g \in \mathcal{L}^2(\varphi).
\end{align*}
$H_n$ is called the \textit{$n$-th Hermite polynomial}.
These polynomials form an orthonormal basis of $\mathcal{L}^{2}(\varphi)$ which implies that any function $g \in \mathcal{L}^2(\varphi)$ can be uniquely expressed as 
$
g(x) = \sum_{n = 0}^{\infty} a_n(g) H_n(x)
$
in the $\mathcal{L}^{2}(\varphi)$-sense where $a_n(g) = \langle g, H_n \rangle_{\mathcal{L}^2(\varphi)}$ are the \textit{Fourier coefficients of $g$}. 
Finally, the index $k_g := \inf \{k \geq 1 \vert a_k(g) \neq 0 \}$ is called the \textit{Hermite rank of $g$}.

\begin{proof}[Proof of \Cref{THM: TRAFO IS OKAY}]
	Stationarity of $Z_\nu$ follows easily from stationarity of $Y$ since we apply a measurable function to $Y$. 
	Next, assume w.l.o.g.~that $u_1 < \ldots < u_J$ such that $g^{-1}(u_1) < \ldots < g^{-1}(u_J)$ where $g^{-1}$ is the inverse of $g$ restricted to $[0, \infty)$.
	Let $g_\nu(x) = \int_{\mathbb{R}} \mathds{1}\{\vert x \vert > g^{-1}(u)\}\, \nu(\mathrm{d}u)$.
	Then, for $n \in \mathbb{Z}$
	\begin{align*}
	Z_{\nu}(n + 1) 
	=
	g_\nu(\vert Y(n + 1) \vert)
	=
	\begin{cases}
	0, &\vert Y(n + 1) \vert \leq g^{-1}(u_1)\\
	w_1, &g^{-1}(u_1) < \vert Y(n + 1) \vert \leq g^{-1}(u_2)\\
	w_1 + w_2, &g^{-1}(u_2) < \vert Y(n + 1) \vert \leq g^{-1}(u_3)\\
	\vdots
	\end{cases} 
	\end{align*}
	Thus, $Z_{\nu}(n + 1)$ is a non-trivial function of $\vert Y(n + 1) \vert$ as long as $g^{-1}(u_1), \ldots, g^{-1}(u_J) \in \text{Im}(Y) = (0, \infty)$.
	This is fulfilled by the assumptions on $g$ and $g^{-1}$.
	Since $Y$ is non-deterministic and Gaussian, Wold's decomposition gives us $Y(k) = \sum_{j \geq 0} a_j \varepsilon_{k - j}$, $k \in \mathbb{Z}$, where $(\varepsilon_{j})_{j \in \mathbb{Z}}$ is a sequence of iid.~standard normal random variables.
	Thus, $Z_{\nu}(n + 1)$ is a non-trivial function of $\varepsilon_{n + 1}$ as well.
	Therefore, $Z_{\nu}(n + 1)$ is not measurable w.r.t.~$\mathcal{F}_{n} := \sigma(\varepsilon_{n}, \varepsilon_{n - 1}, \ldots)$.
	
	If $Z_{\nu}$ were deterministic then
	\begin{align*}
	Z_{\nu}(n + 1)
	= 
	\sum_{j = 1}^{J} w_j \mathds{1}\big\{ 
	\vert Y(n + 1) \vert > g^{-1}(u_j)
	\big\}
	=
	\sum_{j = 1}^{J} w_j \sum_{k \leq n} c_{k, j} 
	\mathds{1}\big\{ 
	\vert Y(k) \vert > g^{-1}(u_j)
	\big\},
	\end{align*}
	where $c_{k, j} \in \mathbb{R}$ are the coefficients of the projection of $Z_{\delta_{u_j}}(n + 1)$ onto its past.
	The RHS is measurable w.r.t.~$\mathcal{F}_{n}$ which is a contradiction.
	Hence, $Z_{\nu}$ is non-deterministic.
	
	Clearly, $g_{\nu} \in \mathcal{L}^{2}(\varphi)$ such that there exists a representation $g_{\nu}(x) = \sum_{k = 0}^{\infty} a_{k, \nu} H_{k}(x)$ based on Hermite polynomials where $a_{k, \nu} = \langle g_{\nu}, H_k \rangle_{\mathcal{L}^2(\varphi)}$.
	Since $g_{\nu}$ is not a constant function, it has finite Hermite rank.
	Therefore, we can apply Theorem 6.3.5 from \cite{Samorodnitsky.2016}. 
	%	This theorem states that if $Y_1$ is a unit-variance centered Gaussian time series whose covariance function $\gamma_{1}(n)$ converges to zero as $n \rightarrow \infty$, then the covariance function $\gamma_{2}(n)$ of a subordinated Gaussian time series $Z_1 = g(Y_1)$ fulfills
	%	\begin{align*}
	%	\lim\limits_{n \rightarrow \infty} \frac{\gamma_{2}(n)}{\gamma_{1}^{k_g}(n)} = \frac{a^2_{k_g}(g)}{k_g!},
	%	\end{align*}
	%	where $k_g$ is the Hermite rank of $g$ and $a_{k_g}(g)$ is the corresponding Fourier coefficient. 
	Thus, the covariance function of $Z_\nu$ is also regularly varying and fulfills condition \eqref{eq: slowly varying convergence assumption}.
	
	Since the spectral density $f_Y$ of $Y$ exists, we can now prove the existence of the spectral density of $Z_\nu$ by Theorem 6.3.4 from \cite{Samorodnitsky.2016}. 
	This theorem implies that the spectral distribution $F_{Z_\nu}$ of $Z_\nu$ is given by
	\begin{align*}
	F_{Z_\nu} = \sum_{m = 1}^{\infty} \frac{a_m^2(g)}{m!} F_{Y}^{\ast, m, f},
	\end{align*}
	where $F_{Y}$ is the spectral distribution of $Y$ and $F_{Y}^{\ast, m, f}$ is its $m$-th folded convolution, i.e.
	\begin{align*}
	F_{Y}^{\ast, m, f}(A) = F_{Y} \times \ldots \times F_{Y} \Big(
	\{ 
	(x_1, \ldots, x_m) \in (-\pi, \pi]^m : x_1 + \ldots + x_m \in A\mod 2\pi
	\}
	\Big),
	\end{align*}
	for all Borel subsets $A$ of $(-\pi, \pi]$. 
	
	Since $Y$ has a spectral density $f_Y$, the $m$-th folded convolution $F_Y^{\ast, m, f}$ of the spectral distribution $F_Y$ fulfills
	\begin{align*}
	F_Y^{\ast, m, f}(A) 
	= 
	\int_{-\pi}^{\pi} \cdots \int_{-\pi}^{\pi} 
	\mathds{1}\Big\{
	x_1 + \ldots x_m \in A \mod 2\pi
	\Big\}
	\bigg(
	\prod_{i = 1}^{m} f_Y(x_i)
	\bigg)\
	\mathrm{d}x_1 \ldots \mathrm{d}x_m.
	\end{align*}
	Thus, $F_Y^{\ast, m, f}(A)$ is absolutely continuous w.r.t. to the $m$-dimensional Lebesgue measure for every $m \in \mathbb{N}$. 
	Using the translation invariance of the Lebesgue measure and Fubini's theorem it is easy to show that $$\big\{ 
	(x_1, \ldots, x_m) \in (-\pi, \pi]^m : x_1 + \ldots + x_m \in A\mod 2\pi
	\big\}$$ has an $m$-dimensional Lebesgue null measure if $A$ is a one-dimensional null set.
	Consequently, $F_Y^{\ast, m, f}(A) = 0$ for any Lebesgue null set $A$ and for all $m \in \mathbb{N}$.
	By Theorem 6.3.4 from \cite{Samorodnitsky.2016}, the same holds true for the spectral distribution $F_{Z_\nu}$ of $Z_\nu$.
	Thus, $F_{Z_\nu}$ has a density w.r.t. to the Lebesgue measure which means that $Z_\nu$ has a spectral density.
	$\blacksquare$
\end{proof}

\end{document}